\documentclass[11pt]{article} 
\usepackage{fullpage}
\usepackage{amsmath,amsthm,amssymb}
\usepackage{algorithm}
\usepackage{graphicx}

\newtheorem{assumption}{Assumption}

\newtheorem{lemma}{Lemma}
\newtheorem{proposition}{Proposition}
\newtheorem{theorem}{Theorem}
\newtheorem{corollary}{Corollary}
\theoremstyle{remark}\newtheorem*{remark}{Remark}

\newcommand{\reals}{\mathbb{R}}
\newcommand{\minimize}{\mathop{\mathrm{minimize}{}}}
\newcommand{\maximize}{\mathop{\mathrm{maximize}{}}}

\newcommand{\argmin}{\mathop{\mathrm{arg\,min}{}}}

\newcommand{\dom}{\mathrm{dom\,}}
\newcommand{\eqdef}{\stackrel{\mathrm{def}}{=}}
\newcommand{\E}{\mathbf{E}}
\newcommand{\bE}{\mathbf{E}}

\newcommand{\tz}{\tilde{z}}

\newcommand{\ba}{\begin{array}}
\newcommand{\ea}{\end{array}}
\newcommand{\beq}{\begin{equation}}
\newcommand{\eeq}{\end{equation}}
\newcommand{\beqa}{\begin{eqnarray}}
\newcommand{\eeqa}{\end{eqnarray}}
\newcommand{\beqas}{\begin{eqnarray*}}
\newcommand{\eeqas}{\end{eqnarray*}}
\newcommand{\bi}{\begin{itemize}}
\newcommand{\ei}{\end{itemize}}

\newcommand{\nn}{\nonumber}
\def\eqref#1{(\ref{#1})}

\def\tf{{\tilde f}}
\def\tpsi{{\tilde \Psi}}
\def\tL{{\tilde L}}

\title{An Accelerated Proximal Coordinate Gradient Method
and its Application to Regularized Empirical Risk Minimization}
\author{
Qihang Lin\thanks{
    Tippie College of Business,
    The University of Iowa,
    Iowa City, IA 52242, USA.
    Email: \texttt{qihang-lin@uiowa.edu}.
}
\and Zhaosong Lu\thanks{
    Department of Mathematics,
    Simon Fraser University,
    Surrey, BC V3T 0A3, Canada.
    Email: \texttt{zhaosong@sfu.ca}.
}
\and Lin Xiao\thanks{
    Machine Learning Groups,
    Microsoft Research,
    Redmond, WA 98052, USA.
    Email: \texttt{lin.xiao@microsoft.com}.
}
}

\begin{document}
\maketitle

\begin{abstract}
We consider the problem of minimizing the sum of two convex functions:
one is smooth and given by a gradient oracle, and the other is separable
over blocks of coordinates and has a simple known structure over each block.
We develop an accelerated randomized proximal coordinate gradient (APCG) 
method for minimizing such convex composite functions. 
For strongly convex functions, our method achieves faster linear convergence 
rates than existing randomized proximal coordinate gradient methods.
Without strong convexity, our method 
enjoys accelerated sublinear convergence rates.
We show how to apply the APCG method to solve the regularized 
empirical risk minimization (ERM) problem, and devise efficient implementations
that avoid full-dimensional vector operations.
For ill-conditioned ERM problems, our method obtains improved convergence rates
than the state-of-the-art stochastic dual coordinate ascent (SDCA) method.
\end{abstract}

\section{Introduction}

Coordinate descent methods have received extensive attention in recent years
due to its potential for solving large-scale optimization problems 
arising from machine learning and other applications
(e.g., \cite{Platt99,HCLKS08,WuLange08,LiOsher09,
WenGoldfardScheinberg12,QinScheinbergGoldfarb13}).
In this paper, we develop an accelerated proximal coordinate gradient (APCG)
method for solving problems of the following form:
\begin{equation}\label{eqn:composite-min}
\minimize_{x\in\reals^N} \quad \bigl\{ F(x) \eqdef f(x) + \Psi(x) \bigr\},
\end{equation}
where $f$ and $\Psi$
are proper and lower semicontinuous convex functions
\cite[Section~7]{Rockafellar70book}.
Moreover, we assume that $f$ is differentiable on $\reals^N$,
and $\Psi$ has a block separable structure, i.e., 
\begin{equation}\label{eqn:Psi}
\Psi(x) = \sum_{i=1}^n \Psi_i(x_i) ,
\end{equation}
where each $x_i$ denotes a sub-vector of~$x$ with cardinality~$N_i$,
and the collection $\{x_i :i=1,\ldots,n\}$ form a partition of the components 
of~$x$.
In addition to the capability of modeling nonsmooth terms 
such as $\Psi(x)=\lambda\|x\|_1$, this model also includes optimization 
problems with block separable constraints. More specifically, each 
block constraints $x_i\in C_i$, where $C_i$ is a closed convex set,
can be modeled by an indicator function defined as
$\Psi_i(x_i)=0$ if $x_i\in C_i$ and $\infty$ otherwise.

At each iteration, coordinate descent methods choose one block of coordinates 
$x_i$ 
to sufficiently reduce the objective value while keeping other blocks fixed. 
In order to exploit the known structure of each $\Psi_i$, a \emph{proximal} 
coordinate gradient step can be taken \cite{RichtarikTakac12}.
To be more specific, given the current iterate $x^{(k)}$, 
we pick a block~$i_k\in\{1,\ldots,n\}$ 
and solve a block-wise proximal subproblem in the form of
\begin{equation}\label{eqn:rpcg-prox}
    h^{(k)}_{i_k} = \argmin_{h\in\Re^{N_{i_k}}} \left\{ 
        \langle \nabla_{i_k} f(x^{(k)}), h\rangle
        + \frac{L_{i_k}}{2}\|h\|^2 + \Psi_{i_k}(x^{(k)}_{i_k} + h) \right\},
\end{equation}
and then set the next iterate as
\begin{equation}\label{eqn:rpcg-update}
    x^{(k+1)}_i =\left\{\begin{array}{ll}
        x^{(k)}_{i_k} + h^{(k)}_{i_k}, & \mbox{if}~i=i_k, \\[1ex]
        x^{(k)}_i, & \mbox{if}~i\neq i_k,
    \end{array} \right.
    \qquad i=1,\ldots,n.
\end{equation}
Here $\nabla_i f(x)$ denotes the \emph{partial gradient} of~$f$ with respect
to~$x_i$, and~$L_i$ is the Lipschitz constant of the partial gradient
(which will be defined precisely later).

One common approach for choosing such a block is the
\emph{cyclic} scheme. The global and local convergence properties of the 
cyclic coordinate descent method have been studied in, e.g.,
\cite{Tseng01,LuoTseng02,SahaTewari13,BeTe13,HongWangLuo13}.
Recently, \emph{randomized} strategies for choosing the block to update became
more popular 
\cite{SST:09,LeventhalLewis10,Nesterov12rcdm,RichtarikTakac12}.
In addition to its theoretical benefits 
(randomized schemes are in general easier to analyze than the cyclic scheme),
numerous experiments have demonstrated that randomized coordinate descent
methods are very powerful for solving large-scale machine learning problems
\cite{ChangHsiehLin08,HCLKS08,SST:09,SSZhang13SDCA}. 
Their efficiency can be further improved with
parallel and distributed implementations 
\cite{Shotgun11icml, RichtarikTakac12bigdata, RichtarikTakac13distributed,
NecoaraClipici13,LiuWrightRe14}.
Randomized block coordinate descent methods have also been proposed and 
analyzed for solving problems with coupled linear constraints 
\cite{TsengYun09linear,NecoaraPatrascu14}
and a class of structured nonconvex optimization problems 
(e.g., \cite{LuXiao13nm,PatrascuNecoara13}). 
Coordinate descent methods with more general schemes of choosing
the block to update have also been studied; 
see, e.g., \cite{BertsekasTsitsiklis89,TsengYun09,Wright12}.

Inspired by the success of accelerated full gradient methods
\cite{Nesterov04book,BeT:09,Tse:08,Nesterov13composite},
several recent work extended Nesterov's acceleration technique
to speed up randomized coordinate descent methods. 
In particular, Nesterov \cite{Nesterov12rcdm} developed 
an accelerated randomized coordinate gradient method for minimizing
unconstrained smooth functions, which corresponds to 
the case of $\Psi(x)\equiv 0$ in~\eqref{eqn:composite-min}.
Lu and Xiao \cite{LuXiao13analysis} gave a sharper convergence analysis 
of Nesterov's method using a randomized estimate sequence framework,
and Lee and Sidford \cite{LeeSidford13} developed extensions using
weighted random sampling schemes.
Accelerated coordinate gradient methods have also been used to speed up
the solution of linear systems \cite{LeeSidford13,LiuWright13}.
More recently, Fercoq and Richt\'{a}rik \cite{FercoqRichtarik13} 
proposed an APPROX (Accelerated, Parallel and PROXimal) coordinate descent 
method for solving the more general problem~\eqref{eqn:composite-min} 
and obtained accelerated sublinear convergence rate, 
but their method cannot exploit the strong convexity of the objective function
to obtain accelerated linear rates.

In this paper, we propose a general APCG method that achieves
accelerated linear convergence rates when the objective function is strongly
convex.
Without the strong convexity assumption, our method recovers a special case
of the APPROX method \cite{FercoqRichtarik13}.
Moreover, we show how to apply the APCG method to solve the 
regularized empirical risk minimization (ERM) problem, and devise efficient 
implementations that avoid full-dimensional vector operations.
For ill-conditioned ERM problems, our method obtains improved convergence rates
than the state-of-the-art stochastic dual coordinate ascent (SDCA) method
\cite{SSZhang13SDCA}.

\subsection{Outline of paper}
This paper is organized as follows. 
The rest of this section introduces some notations and 
state our main assumptions. 
In Section~\ref{sec:apcg}, we present the general APCG method
and our main theorem on its convergence rate. 
We also give two simplified versions of APCG depending on whether or not
the function~$f$ is strongly convex, and explain how to exploit 
strong convexity in~$\Psi$.
Section~\ref{sec:analysis} is devoted to the convergence analysis that
proves our main theorem.
In Section~\ref{sec:effi-impl}, we derive equivalent implementations
of the APCG method that can avoid full-dimensional vector operations.

In Section~\ref{sec:erm}, we apply the APCG method to solve the dual of
the regularized ERM problem and give the corresponding complexity results. 
We also explain how to recover primal solutions to guarantee the same rate
of convergence for the primal-dual gap. 
In addition, we present numerical experiments to demonstrate 
the performance of the APCG method.


\subsection{Notations and assumptions}

For any partition of $x\in\reals^N$ into $\{x_i\in\reals^{N_i}:i=1,\ldots,n\}$
with $\sum_{i=1}^n N_i = N$,
there is an $N\times N$ permutation matrix $U$
partitioned as $U=[U_1 \cdots U_n]$, where $U_i\in\reals^{N\times N_i}$, such that
\[
x = \sum_{i=1}^n U_i x_i, \quad
\mbox{and}\quad x_i = U_i^T x, \quad i=1,\ldots,n.
\]
For any $x\in\reals^N$, the \emph{partial gradient} of~$f$ with respect to~$x_i$
is defined as
\[
\nabla_i f(x) = U_i^T \nabla f(x), \quad i=1,\ldots,n.
\]
We associate each subspace $\reals^{N_i}$, for $i=1,\ldots,n$,
with the standard Euclidean norm, denoted $\|\cdot\|_2$.
We make the following assumptions which are standard in the literature
on coordinate descent methods 
(e.g., \cite{Nesterov12rcdm,RichtarikTakac12}).

\begin{assumption}\label{asmp:coord-smooth}
The gradient of function~$f$ is block-wise Lipschitz continuous
with constants $L_i$, i.e.,
\[
\|\nabla_i f(x+U_i h_i) - \nabla_i f(x) \|_2 \leq L_i\|h_i\|_2,
\quad \forall\, h_i\in\reals^{N_i}, \quad i=1,\ldots,n, \quad x\in\reals^N.
\]
\end{assumption}

An immediate consequence of Assumption~\ref{asmp:coord-smooth} is
(see, e.g., \cite[Lemma~1.2.3]{Nesterov04book})
\begin{equation} \label{eqn:lip-ineq}
f(x+U_i h_i) \leq f(x) + \langle \nabla_i f(x), h_i\rangle
+ \frac{L_i}{2}\|h_i\|_2^2,
\quad \forall\, h_i\in\reals^{N_i}, \quad i=1,\ldots,n, \quad x\in\reals^N.
\end{equation}
For convenience, we define the following weighted norm 
in the whole space $\reals^N$:
\begin{eqnarray} \label{eqn:L-norm}
\|x\|_L &=& \biggl(\sum_{i=1}^n L_i \|x_i\|_2^2\biggr)^{1/2},
\quad\forall\,x\in\reals^N.
\end{eqnarray}

\begin{assumption}\label{asmp:strong-convex}
There exists $\mu\geq0$ such that 
for all $y\in\reals^N$ and $x\in\dom(\Psi)$,
\[
f(y) \geq f(x) + \langle \nabla f(x), y-x\rangle + \frac{\mu}{2}
\|y-x\|_L^2.
\]
\end{assumption}
The \emph{convexity parameter} of~$f$ with respect to the norm $\|\cdot\|_L$ 
is the largest $\mu$ such that the above inequality holds. 
Every convex function satisfies Assumption~\ref{asmp:strong-convex} 
with $\mu=0$.
If $\mu>0$, then the function~$f$ is called \emph{strongly} convex.

\begin{remark}
Together with~\eqref{eqn:lip-ineq} and the definition of
$\|\cdot\|_L$ in~\eqref{eqn:L-norm}, 
Assumption~\ref{asmp:strong-convex} implies $\mu\leq 1$.
\end{remark}

\section{The APCG method} 
\label{sec:apcg}
In this section we describe the general APCG method, 
and its two simplified versions under different assumptions
(whether or not the objective function is strong convex).
We also present our main theorem on the convergence rates of the APCG method.

We first explain the notations used in our algorithms.
The algorithms proceed in iterations, with~$k$ being the iteration counter.
Lower case letters $x$, $y$, $z$ represent vectors in the full space $\reals^N$,
and $x^{(k)}$, $y^{(k)}$ and $z^{(k)}$ are their values at the $k$th iteration.
Each block coordinate is indicated with a subscript, for example, 
$x^{(k)}_i$ represent the value of the $i$th block of the vector $x^{(k)}$.
The Greek letters $\alpha$, $\beta$, $\gamma$ are scalars,
and $\alpha_k$, $\beta_k$ and $\gamma_k$ represent their values
at iteration~$k$.
For scalars, a superscript represents the power exponent;
for example, $n^2$, $\alpha_k^2$ denotes the squares of $n$ 
and $\alpha_k$ respectively.

\begin{algorithm}[t]
\caption{The APCG method}
\label{alg:apcg}
\textbf{input:} $x^{(0)}\in\dom(\Psi)$ and convexity parameter $\mu\geq 0$.\\[0.5ex]
\textbf{initialize:} set $z^{(0)}=x^{(0)}$
    and choose $0 < \gamma_0 \in [\mu, 1]$. \\[0.5ex]
\textbf{iterate:} repeat for $k=0,1,2,\ldots$ 
\begin{enumerate}  \itemsep 0pt 
\item Compute $\alpha_k\in(0, \frac1n]$ from the equation
\begin{equation} \label{alphak}
n^2\alpha_k^2 = \left(1-\alpha_k\right)\gamma_k + \alpha_k \mu,
\end{equation}
and set
\begin{equation} \label{gammak}
\gamma_{k+1} = (1-\alpha_k) \gamma_k + \alpha_k \mu, \qquad
\beta_k = \frac{\alpha_k\mu}{\gamma_{k+1}} .
\end{equation}
\item Compute $y^{(k)}$ as 
\begin{equation} \label{yk}
y^{(k)} ~=~  \frac{1}{\alpha_k \gamma_k+\gamma_{k+1}} \left(\alpha_k \gamma_k z^{(k)} + \gamma_{k+1} x^{(k)}\right).
\end{equation}
\item Choose $i_k\in\{1,\ldots,n\}$ uniformly at random and compute
\[
    z^{(k+1)} = \argmin_{x\in\reals^N} \Bigl\{ \frac{n\alpha_k}{2} 
    \bigl\|x - (1-\beta_k) z^{(k)} - \beta_k  y^{(k)} \bigr\|_L^2
    + \langle \nabla_{i_k} f(y^{(k)}), x_{i_k}\rangle
    + \Psi_{i_k}(x_{i_k}) \Bigr\}.
\]
\item Set
\begin{equation}\label{xk}
 x^{(k+1)} = y^{(k)} + n\alpha_k(z^{(k+1)}-z^{(k)}) + \frac{\mu}{n}(z^{(k)}-y^{(k)}) .
\end{equation}
\end{enumerate}
\end{algorithm}

The general APCG method is given as Algorithm~\ref{alg:apcg}.
At each iteration~$k$, the APCG method picks a random coordinate
$i_k\in\{1,\ldots,n\}$ and generates $y^{(k)}$, $x^{(k+1)}$ and $z^{(k+1)}$. 
One can observe that $x^{(k+1)}$ and $z^{(k+1)}$ depend on the realization 
of the random variable
\[
\xi_k = \{i_0, i_1,\ldots, i_k\},
\]
while $y^{(k)}$ is independent of~$i_k$ and 
only depends on $\xi_{k-1}$. 

To better understand this method, we make the following observations.
For convenience, we define
\begin{equation}\label{eqn:full-z-update}
 \tz^{(k+1)} = \argmin_{x\in\reals^N} \Bigl\{ 
     \frac{n\alpha_k}{2} \bigl\|x - (1-\beta_k) z^{(k)} - \beta_k y^{(k)} \bigr\|_L^2
    + \langle \nabla f(y^{(k)}), x-y^{(k)}\rangle + \Psi(x) \Bigr\},
\end{equation}
which is a full-dimensional update version of Step~3.
One can observe that~$z^{(k+1)}$ is updated as
\begin{equation}\label{eqn:z-update}
    z^{(k+1)}_i = \left\{\begin{array}{ll}
        \tz^{(k+1)}_i & \mathrm{if}~i=i_k,\\ [5pt]
        (1-\beta_k) z^{(k)}_i + \beta_k y^{(k)}_i & \mathrm{if}~i\neq i_k.
    \end{array} \right.
\end{equation}
Notice that from \eqref{alphak}, \eqref{gammak}, \eqref{yk} and \eqref{xk} 
we have
\[
x^{(k+1)} = y^{(k)} +  n\alpha_k \left(z^{(k+1)} - (1-\beta_k) z^{(k)} - \beta_k y^{(k)}\right),
\]
which together with \eqref{eqn:z-update} yields
\begin{equation}\label{eqn:x-update}
    x^{(k+1)}_i = \left\{\begin{array}{ll}
        y^{(k)}_i + n \alpha_k \left( z^{(k+1)}_i - z^{(k)}_i\right) + \frac{\mu}{n} \left( z^{(k)}_i-y^{(k)}_i \right)  & \mathrm{if}~i=i_k, \\ [5pt]
        y^{(k)}_i & \mathrm{if}~i\neq i_k.
    \end{array} \right.
\end{equation}
That is, in Step~4, we only need to update the block coordinates 
$x^{(k+1)}_{i_k}$ as in~\eqref{eqn:x-update}
and set the rest to be $y^{(k)}_i$.

We now state an expected-value type of convergence rate for the APCG method. 

\begin{theorem}\label{thm:apcg-rate}
Suppose Assumptions~\ref{asmp:coord-smooth} and~\ref{asmp:strong-convex} hold.
Let $F^\star$ be the optimal value of problem~(\ref{eqn:composite-min}),
and $\{x^{(k)}\}$ be the sequence generated by the APCG method. Then, for any $k\geq 0$,
there holds:
\[
\bE_{\xi_{k-1}} [F(x^{(k)})] - F^\star ~\leq~ \min \left\{ \left(1-\frac{\sqrt{\mu}}{n}\right)^k,
~\left(\frac{2n}{2n+k\sqrt{\gamma_0}}\right)^2 \right\}
\left(F(x^{(0)})-F^\star+\frac{\gamma_0}{2}R^2_0\right),
\]
where
\beq \label{R0}
R_0 \eqdef \min\limits_{x^\star\in X^\star} \|x^{(0)}-x^\star\|_L,
\eeq
and $X^\star$ is the set of optimal solutions of problem~(\ref{eqn:composite-min}).
\end{theorem}

For $n=1$, our results in Theorem~\ref{thm:apcg-rate} match exactly the 
convergence rates of 
the accelerated full gradient method in~\cite[Section~2.2]{Nesterov04book}.
For $n>1$, our results improve upon the convergence rates 
of the randomized proximal coordinate gradient method described
in~\eqref{eqn:rpcg-prox} and~\eqref{eqn:rpcg-update}.
More specifically, if the block index $i_k\in\{1,\ldots,n\}$ is chosen 
uniformly at random, then the analysis in 
\cite{RichtarikTakac12,LuXiao13analysis} states that
the convergence rate of~\eqref{eqn:rpcg-prox} and~\eqref{eqn:rpcg-update}
is on the order of 
\[
    O\left(\min \left\{ \left(1-\frac{\mu}{n}\right)^k, ~\frac{n}{n+k} \right\}
\right) .
\]
Thus we obtain both accelerated linear rate for strongly convex functions 
($\mu>0$) and accelerated sublinear rate for non-strongly convex functions
($\mu=0$).
To the best of our knowledge, this is the first time that such an accelerated
linear convergence rate is obtained for solving the general class of 
problems~\eqref{eqn:composite-min}
using coordinate descent type of methods.

The proof of Theorem~\ref{thm:apcg-rate} is given in Section~\ref{sec:analysis}.
Next we give two simplified versions of the APCG method, 
for the special cases of $\mu>0$ and $\mu=0$, respectively.

\subsection{Two special cases}

For the strongly convex case with $\mu>0$, we can initialize 
Algorithm~\ref{alg:apcg} with the parameter $\gamma_0=\mu$, 
which implies $\gamma_k=\mu$ and $\alpha_k=\beta_k=\sqrt{\mu}/n$
for all $k\geq 0$.
This results in Algorithm~\ref{alg:apcg-mu}.
As a direct corollary of Theorem~\ref{thm:apcg-rate}, 
Algorithm~\ref{alg:apcg-mu} enjoys an accelerated linear convergence rate:
\[
    \bE_{\xi_{k-1}} [F(x^{(k)})] - F^\star 
~\leq~ \left(1-\frac{\sqrt{\mu}}{n}\right)^k
\left(F(x^{(0)})-F^\star+\frac{\mu}{2}\|x^{(0)}-x^\star\|_L^2\right),
\]
where $x^\star$ is the unique solution of~\eqref{eqn:composite-min}
under the strong convexity assumption.

\begin{algorithm}[t]
\caption{APCG with $\gamma_0=\mu>0$}
\label{alg:apcg-mu}
\textbf{input:} $x^{(0)}\in\dom(\Psi)$ and convexity parameter $\mu>0$.\\[0.5ex]
\textbf{initialize:} set $z^{(0)}=x^{(0)}$ 
    and $\alpha=\frac{\sqrt{\mu}}{n}$. \\[0.5ex]
\textbf{iterate:} repeat for $k=0,1,2,\ldots$ 
and repeat for $k=0,1,2,\ldots$
\begin{enumerate} \itemsep 0pt
\item Compute $y^{(k)} = \frac{x^{(k)} + \alpha z^{(k)}}{1+\alpha}$.
\item Choose $i_k\in\{1,\ldots,n\}$ uniformly at random and compute\\[1ex]
\[
    z^{(k+1)} = \argmin_{x\in\reals^N} \Bigl\{ 
    \frac{n\alpha}{2} \bigl\|x - (1-\alpha)z^{(k)} - \alpha y^{(k)} \bigr\|_L^2 
    + \langle \nabla_{i_k} f(y^{(k)}), x_{i_k}-y^{(k)}_{i_k}\rangle 
    + \Psi_{i_k}(x_{i_k}) \Bigr\}.
\]
\item Set $x^{(k+1)} = y^{(k)}+n\alpha(z^{(k+1)}-z^{(k)})+n\alpha^2(z^{(k)}-y^{(k)})$.
\end{enumerate}
\end{algorithm}

\begin{algorithm}[t]
\caption{APCG with $\mu=0$}
\label{alg:apcg-mu=0}
\textbf{Input:} $x^{(0)}\in\dom(\Psi)$. \\[0.5ex]
\textbf{Initialize:} set $z^{(0)}=x^{(0)}$
    and choose $\alpha_{-1} \in (0, \frac{1}{n}]$. \\[0.5ex]
\textbf{Iterate:} repeat for $k=0,1,2,\ldots$ 
\begin{enumerate}  \itemsep 0pt 
\item Compute 
$
\alpha_k = \frac{1}{2}\left(\sqrt{\alpha_{k-1}^4+4\alpha_{k-1}^2}-\alpha_{k-1}^2\right) .
$
\item Compute 
$
y^{(k)} ~=~ (1-\alpha_k) x^{(k)} + \alpha_k z^{(k)} 
$.
\item Choose $i_k\in\{1,\ldots,n\}$ uniformly at random and compute\\
\[
    z^{(k+1)}_{i_k} = \argmin_{x\in\reals^N}\Bigl\{ \frac{n\alpha_k L_{i_k}}{2} 
    \bigl\|x - z^{(k)}_{i_k} \bigr\|_2^2
    + \langle \nabla_{i_k} f(y^{(k)}), x-y^{(k)}_{i_k}\rangle
    + \Psi_{i_k}(x) \Bigr\}.
\]
and set $z^{(k+1)}_i = z^{(k)}_i$ for all  $i\neq i_k$.
\item Set
$
    x^{(k+1)} = y^{(k)} + n\alpha_k(z^{(k+1)}-z^{(k)}) .
$
\end{enumerate}
\end{algorithm}

Algorithm~\ref{alg:apcg-mu=0} shows the simplified version for $\mu=0$,
which can be applied to problems without strong convexity, or if the
convexity parameter $\mu$ is unknown.
According to Theorem~\ref{thm:apcg-rate}, Algorithm~\ref{alg:apcg-mu=0} has an
accelerated sublinear convergence rate, that is
\[
\bE_{\xi_{k-1}} [F(x^{(k)})] - F^\star ~\leq~ 
\left(\frac{2n}{2n+k n\alpha_0}\right)^2
\left(F(x^{(0)})-F^\star+\frac{n\alpha_0}{2}R^2_0\right).
\]
With the choice of $\alpha_{-1} = 1/\sqrt{n^2-1}$, which implies
$\alpha_0=1/n$,
Algorithm~\ref{alg:apcg-mu=0} reduces to the APPROX method 
\cite{FercoqRichtarik13} with single block update at each iteration 
(i.e., $\tau=1$ in their Algorithm~1).

\subsection{Exploiting strong convexity in $\Psi$} 
\label{sec:strong-psi}

In this section we consider problem \eqref{eqn:composite-min} 
with strongly convex $\Psi$.  
We assume that $f$ and $\Psi$ have convexity parameters $\mu_f \ge 0$ 
and $\mu_{\Psi} >0$, both with respect to the standard Euclidean norm,
denoted $\|\cdot\|_2$.

Let $x^{(0)} \in \dom(\Psi)$ and $s^{(0)}\in\partial \Psi(x^{(0)})$ 
be arbitrarily chosen, and define two functions
\beqas
\tf(x) &\eqdef& f(x) + \Psi(x^{(0)}) + \langle s^{(0)}, x-x^{(0)}\rangle +\frac{\mu_\Psi}{2} \|x-x^{(0)}\|_2^2 \\
\tpsi(x) &\eqdef& \Psi(x) - \Psi(x^{(0)}) - \langle s^{(0)}, x-x^{(0)}\rangle - \frac{\mu_{\Psi}}{2} \|x-x^{(0)}\|_2^2.
\eeqas
One can observe that the gradient of the function~$\tf$ is block-wise 
Lipschitz continuous with constants $\tL_i = L_i+\mu_{\Psi}$
with respect to the norm $\|\cdot\|_2$.
The convexity parameter of~$\tf$ with respect to the norm
$\|\cdot\|_{\tL}$ defined in~\eqref{eqn:L-norm} is
\beq \label{mu}
\mu: = \frac{\mu_f+\mu_\Psi}{\max\limits_{1\le i\le n} \{L_i+\mu_{\Psi}\}}.
\eeq
In addition, $\tpsi$ is a block separable convex function which 
can be expressed as $\tpsi(x) = \sum^n_{i=1} \tpsi_i(x_i)$, where
\[
\tpsi_i(x_i) ~=~ \Psi_i(x_i) - \Psi_i(x^0_i) - \langle s^0_i, x_i-x^0_i\rangle - \frac{\mu_{\Psi}}{2} \|x_i-x^0_i\|_2^2, \qquad i=1,\ldots,n.
\]
As a result of the above definitions, 
we see that problem \eqref{eqn:composite-min} is equivalent to 
\beq \label{ref-prob}
\minimize_{x\in\Re^N} \ \left\{\tf(x) + \tpsi(x) \right\},
\eeq
which can be suitably solved by the APCG method proposed in Section \ref{sec:apcg} with $f$, $\Psi_i$  and $L_i$ replaced by $\tf$, $\tpsi_i$ and $L_i+\mu_{\Psi}$, respectively.
The rate of convergence of APCG applied to problem \eqref{ref-prob} 
directly follows from Theorem \ref{thm:apcg-rate}, with~$\mu$ given 
in~\eqref{mu} and the norm $\|\cdot\|_L$ in~\eqref{R0} replaced 
by $\|\cdot\|_{\tL}$.

\section{Convergence analysis}
\label{sec:analysis}

In this section, we prove Theorem~\ref{thm:apcg-rate}.
First we establish some useful properties of the sequences
$\{\alpha_k\}^\infty_{k=0}$ and $\{\gamma_k\}^\infty_{k=0}$ 
generated in Algorithm~\ref{alg:apcg}.
Then in Section~\ref{sec:Phi-hat}, we construct a sequence 
$\{\hat\Psi_k\}_{k=1}^\infty$ to bound the values of $\Psi(x^{(k)})$
and prove a useful property of the sequence.
Finally we finish the proof of Theorem~\ref{thm:apcg-rate}
in Section~\ref{sec:prove-thm1}.

\begin{lemma} \label{alphak-prop}
Suppose $\gamma_0>0$ and $\gamma_0 \in [\mu,1]$ and 
$\{\alpha_k\}^\infty_{k=0}$ and $\{\gamma_k\}^\infty_{k=0}$ are generated 
in Algorithm~\ref{alg:apcg}. 
Then there hold:
\begin{itemize}
\item[(i)] $\{\alpha_k\}^\infty_{k=0}$ and $\{\gamma_k\}^\infty_{k=0}$ 
    are well-defined positive sequences. 
\item[(ii)] $\sqrt{\mu}/n \le \alpha_k \le 1/n$ and $\mu \le \gamma_k \le 1$ for all $k \ge 0$.
\item[(iii)]  $\{\alpha_k\}^\infty_{k=0}$ and $\{\gamma_k\}^\infty_{k=0}$ are non-increasing. 
\item[(iv)] $\gamma_k = n^2 \alpha^2_{k-1}$ for all $k \ge 1$.
\item[(v)]  With the definition of
\begin{equation}\label{eqn:lambda-def}
\lambda_k= \prod^{k-1}_{i=0} (1-\alpha_i),
\end{equation}
we have for all $k\geq 0$,
\[
\lambda_k ~\leq~ \min \left\{ \left(1-\frac{\sqrt{\mu}}{n}\right)^k,
~\left(\frac{2n}{2n+k\sqrt{\gamma_0}}\right)^2 \right\} .
\]
\end{itemize}
\end{lemma}

\begin{proof}
Due to \eqref{alphak} and \eqref{gammak}, statement (iv) always holds 
provided that $\{\alpha_k\}^\infty_{k=0}$ and $\{\gamma_k\}^\infty_{k=0}$ 
are well-defined. 
We now prove statements (i) and (ii) by induction. 
For convenience, Let 
\[
g_\gamma(t) = n^2 t^2-\gamma(1-t)-\mu t.
\]
Since $\mu \le 1$ and $\gamma_0 \in (0,1]$, one can observe that $g_{\gamma_0}(0)=-\gamma_0<0$ 
and 
\[
g_{\gamma_0}\left(\frac1n\right)=1-\gamma_0\left(1-\frac1n\right)-\frac{\mu}{n} \ge 1 - \gamma_0 \ge 0.
\]
These together with continuity of $g_{\gamma_0}$ imply that there exists $\alpha_0 \in (0, 1/n]$ such 
that $g_{\gamma_0}(\alpha_0)=0$, that is, $\alpha_0$ satisfies \eqref{alphak} and is thus well-defined. 
In addition, by statement (iv) and $\gamma_0 \ge \mu$, one can see $\alpha_0 \ge  \sqrt{\mu}/n$. 
Therefore, statements (i) and (ii) hold for $k=0$. 

Suppose that the statements~(i) and~(ii) hold for some $k \ge 0$, that is,
$\alpha_k, \gamma_k>0$, $\sqrt{\mu}/n \le \alpha_k \le 1/n$ and 
$\mu \le \gamma_k \le 1$. 
Using these relations and \eqref{gammak}, one can see that 
$\gamma_{k+1}$ is well-defined and moreover $\mu \le \gamma_{k+1} \le 1$. 
In addition, we have
$\gamma_{k+1} >0$ due to statement (iv) and $\alpha_k>0$. 
Using the fact $\mu \le 1$ 
(see the remark after Assumption~\ref{asmp:strong-convex}), 
$\gamma_0 \in (0,1]$ and a similar argument as above, 
we obtain $g_{\gamma_k}(0)<0$ and $g_{\gamma_k}(1/n) \ge 0$, 
which along with continuity of $g_{\gamma_k}$ imply that there exists 
$\alpha_{k+1} \in (0, 1/n]$ such that $g_{\gamma_k}(\alpha_{k+1})=0$, 
that is, $\alpha_{k+1}$ satisfies \eqref{alphak} and is thus well-defined. 
By statement (iv) and $\gamma_{k+1} \ge \mu$, 
one can see that $\alpha_{k+1} \ge  \sqrt{\mu}/n$. 
This completes the induction and hence statements (i) and (ii) hold. 

Next, we show statement (iii) holds. 
Indeed, it follows from \eqref{gammak} that 
$$\gamma_{k+1} -\gamma_{k} = \alpha_k(\mu-\gamma_k),$$ which together with 
$\gamma_k \ge \mu$ and $\alpha_k >0$ implies that $\gamma_{k+1} \le \gamma_{k}$ and hence 
$\{\gamma_k\}^\infty_{k=0}$ is non-increasing.  Notice from statement (iv) and $\alpha_k>0$ 
that $\alpha_k = \sqrt{\gamma_{k+1}}/n$. It follows that  $\{\alpha_k\}^\infty_{k=0}$ is also 
non-increasing.

Statement (v) can be proved by using the same arguments 
in the proof of \cite[Lemma~2.2.4]{Nesterov04book},
and the details can be found in \cite[Section~4.2]{LuXiao13analysis}.
\end{proof}

\subsection{Construction and properties of $\hat\Psi_k$}
\label{sec:Phi-hat}

Motivated by \cite{FercoqRichtarik13}, we give an explicit expression of
$x^{(k)}$ as a convex combination of the vectors $z^{(0)},\ldots,z^{(k)}$,
and use the coefficients to construct a sequence 
$\{\hat\Psi_k\}_{k=1}^\infty$ to bound $\Psi(x^{(k)})$.

\begin{lemma}\label{lem:combination}
Let the sequences $\{\alpha_k\}^\infty_{k=0}$, $\{\gamma_k\}^\infty_{k=0}$, 
$\{x^{(k)}\}^\infty_{k=0}$ and $\{z^{(k)}\}^\infty_{k=0}$ be generated by 
Algorithm~\ref{alg:apcg}.
Then each $x^{(k)}$ is a convex combination of $z^{(0)},\ldots,z^{(k)}$. 
More specifically, for all $k\geq 0$,
\begin{equation}\label{eqn:combination}
    x^{(k)} = \sum_{l=0}^k \theta^{(k)}_l z^{(l)}, 
\end{equation}
where the constants $\theta^{(k)}_0,\ldots,\theta^{(k)}_k$ are nonnegative
and sum to~$1$.
Moreover, these constants can be obtained recursively by setting
$\theta_0^{(0)}=1$, $\theta^{(1)}_0=1-n\alpha_0$, $\theta^{(1)}_1=n\alpha_0$ 
and for $k\geq 1$,
\begin{equation}\label{eqn:combine-coeff}
    \theta^{(k+1)}_l = \left\{ \begin{array}{ll}
 n\alpha_k        & l=k+1, \\[1ex]
 \left(1-\frac{\mu}{n}\right) \frac{\alpha_k\gamma_k+n\alpha_{k-1}\gamma_{k+1}}
{\alpha_k \gamma_k+\gamma_{k+1}}-\frac{(1-\alpha_k)\gamma_k}{n\alpha_k}         & l=k,\\[1ex]
\left(1-\frac{\mu}{n}\right) \frac{\gamma_{k+1}}{\alpha_k \gamma_k+\gamma_{k+1}} \theta^{(k)}_l  & l=0,\ldots,k-1.
    \end{array} \right.
\end{equation}
\end{lemma}

\begin{proof}
We prove the statements by induction.
First, notice that $x^{(0)}=z^{(0)}=\theta^{(0)}_0 z^{(0)}$.
Using  this relation and \eqref{yk}, we see that $y^{(0)}=z^{(0)}$. 
From~\eqref{xk} and $y^{(0)}=z^{(0)}$, we obtain
\begin{eqnarray}
x^{(1)} &=& y^{(0)} + n\alpha_0\left(z^{(1)}-z^{(0)}\right)
    +\frac{\mu}{n}\left(z^{(0)}-y^{(0)}\right) \nonumber \\
&=& z^{(0)} + n\alpha_0\left(z^{(1)}-z^{(0)}\right) \nonumber \\
&=&  (1- n\alpha_0) z^{(0)} + n\alpha_0 z^{(1)}. \label{x1}
\end{eqnarray}
Since $\alpha_0 \in (0,1/n]$ (Lemma \ref{alphak-prop} (ii)), 
the vector $x^{(1)}$ is a convex combination of $z^{(0)}$ and $z^{(1)}$ 
with the coefficients $\theta^{(1)}_0=1-n\alpha_0$, $\theta^{(1)}_1=n\alpha_0$.
For $k=1$, substituting \eqref{yk} into \eqref{xk} yields
\begin{eqnarray*}
x^{(2)} 
&=& y^{(1)} + n\alpha_1\left(z^{(2)}-z^{(1)}\right)
    +\frac{\mu}{n}\left(z^{(1)}-y^{(1)}\right) \\
&=& \left(1-\frac{\mu}{n}\right) \frac{\gamma_2}{\alpha_1 \gamma_1+\gamma_2} x^{(1)}+ \left[\left(1-\frac{\mu}{n}\right) \frac{\alpha_1\gamma_1}
{\alpha_1 \gamma_1+\gamma_2}-\frac{n^2\alpha_1^2-\alpha_1\mu}{n\alpha_1} \right]z^{(1)} + n\alpha_1 z^{(2)}. 
\end{eqnarray*}
Substituting \eqref{x1} into the above equality, and using
$(1-\alpha_1) \gamma_1 = n^2\alpha^2_1-\alpha_1\mu$ from~\eqref{alphak},
we get
\beq \label{x2}
x^{(2)} =  \underbrace{\left(1-\frac{\mu}{n}\right) \frac{\gamma_2 (1-n\alpha_0)}{\alpha_1 \gamma_1+\gamma_2}}_{\theta^{(2)}_0} z^{(0)}+ \underbrace{ \left[\left(1-\frac{\mu}{n}\right) \frac{\alpha_1\gamma_1+n\alpha_0\gamma_2}
{\alpha_1 \gamma_1+\gamma_2} - \frac{(1-\alpha_1)\gamma_1}{n\alpha_1} \right]}_{\theta^{(2)}_1}z^{(1)} + \underbrace{n\alpha_1}_{\theta^{(2)}_2} z^{(2)}. 
\eeq
From the definition of $\theta^{(2)}_1$ in the above equation, 
we observe that
\beqas
\theta^{(2)}_1 &=& \left(1-\frac{\mu}{n}\right) \frac{\alpha_1\gamma_1+n\alpha_0\gamma_2}
{\alpha_1 \gamma_1+\gamma_2}-\frac{(1-\alpha_1)\gamma_1}{n\alpha_1} \\
&=& \left(1-\frac{\mu}{n}\right) \frac{\alpha_1\gamma_1(1-n\alpha_0)+n\alpha_0(\alpha_1\gamma_1+\gamma_2)}{\alpha_1 \gamma_1+\gamma_2}
-\frac{n^2\alpha^2_1-\alpha_1\mu}{n\alpha_1} \\ [6pt]
&=&  \frac{\alpha_1\gamma_1}{\alpha_1 \gamma_1+\gamma_2}\left(1-\frac{\mu}{n}\right) (1-n\alpha_0) + \left(1-\frac{\mu}{n}\right)n\alpha_0  
-n \alpha_1 + \frac{\mu}{n} \\ [6pt]
&=&  \frac{\alpha_1\gamma_1}{\alpha_1 \gamma_1+\gamma_2}\left(1-\frac{\mu}{n}\right) (1-n\alpha_0) + \left(1-\frac{\mu}{n}\right)n(\alpha_0-\alpha_1) 
+ \frac{\mu}{n} (1-n\alpha_1). 
\eeqas
From the above expression, and using the facts $\mu \le 1$, 
$\alpha_0 \ge \alpha_1$, $\gamma_k \ge 0$ and  
$0 \le \alpha_k \le 1/n$ (Lemma~\ref{alphak-prop}), 
we conclude that $\theta^{(2)}_1\geq 0$.
Also considering the definitions of $\theta^{(2)}_0$ and $\theta^{(2)}_2$ 
in~\eqref{x2}, we conclude that $\theta^{(2)}_l \ge 0$ for $0\le l\le 2$. 
In addition, one can observe from \eqref{yk}, \eqref{xk} and \eqref{x1} that 
$x^{(1)}$ is an affine combination of $z^{(0)}$ and $z^{(1)}$, 
$y^{(1)}$ is an affine combination of $z^{(1)}$ and $x^{(1)}$, 
and $x^{(2)}$ is an affine combination of $y^{(1)}$, $z^{(1)}$ and $z^{(2)}$. 
It is known that substituting one affine combination into another yields a new 
affine combination. Hence, the combination given in \eqref{x2} must be affine, 
which together with $\theta^{(2)}_l \ge 0$ for $0\le l\le 2$ implies that 
it is also a convex combination.

Now suppose the recursion~\eqref{eqn:combine-coeff} holds for some $k\geq 1$.
Substituting \eqref{yk} into \eqref{xk}, we obtain that
\begin{eqnarray*}
x^{(k+1)}  =  \left(1-\frac{\mu}{n}\right) \frac{\gamma_{k+1}}{\alpha_k \gamma_k+\gamma_{k+1}} x^{(k)}+ \left[\left(1-\frac{\mu}{n}\right) \frac{\alpha_k\gamma_k}
{\alpha_k \gamma_k+\gamma_{k+1}}-
\frac{(1-\alpha_k)\gamma_k}{n\alpha_k} \right]z^{(k)} + n\alpha_k z^{(k+1)}. 
\end{eqnarray*}
Further, substituting 
$x^{(k)}=n\alpha_{k-1} z^{(k)} +\sum_{l=0}^{k-1} \theta^{(k)}_l z^{(l)}$ 
(the induction hypothesis) into the above equation gives 
\begin{eqnarray}
x^{(k+1)} &=&
 \sum_{l=0}^{k-1} \underbrace{\left(1-\frac{\mu}{n}\right) \frac{\gamma_{k+1}}{\alpha_k \gamma_k+\gamma_{k+1}} 
 \theta^{(k)}_l}_{\theta^{(k+1)}_l} z^{(l)} + \underbrace{\left[\left(1-\frac{\mu}{n}\right) \frac{\alpha_k\gamma_k+n\alpha_{k-1}\gamma_{k+1}}
{\alpha_k \gamma_k+\gamma_{k+1}}-
\frac{(1-\alpha_k)\gamma_k}{n\alpha_k} \right]}_{\theta^{(k+1)}_k}z^{(k)} \nn \\
&& + \underbrace{n\alpha_k}_{\theta^{(k+1)}_{k+1}} z^{(k+1)}.
\label{eqn:combination-k+1}
\end{eqnarray}
This gives the form of \eqref{eqn:combination} and \eqref{eqn:combine-coeff}.
In addition, by the induction hypothesis, 
$x^{(k)}$ is an affine combination of $z^{(0)},\ldots,z^{(k)}$. 
Also, notice from \eqref{yk} and \eqref{xk} that 
$y^{(k)}$ is an affine combination of $z^{(k)}$ and $x^{(k)}$, and
$x^{(k+1)}$ is an affine combination of $y^{(k)}$, $z^{(k)}$ and $z^{(k+1)}$. 
Using these facts and a similar argument as for $x^{(2)}$, it 
follows that the combination \eqref{eqn:combination-k+1} must be affine. 

Finally, we claim $\theta^{(k+1)}_l \ge 0$ for all $l$. 
Indeed, we know from Lemma \ref{alphak-prop} that 
$\mu \le 1$, $\alpha_k\ge 0$, $\gamma_k \ge 0$. 
Also, $\theta^{(k)}_l \ge 0$ due to the induction hypothesis. 
It follows that $\theta^{(k+1)}_l \ge 0$ for all $l \neq k$. 
It remains to show that $\theta^{(k+1)}_k \ge 0$. 
To this end,  we again use \eqref{alphak} to obtain 
$(1-\alpha_k) \gamma_k = n^2\alpha^2_k-\alpha_k\mu$,
and use~\eqref{eqn:combine-coeff} 
and a similar argument as for $\theta^{(2)}_1$ to rewrite 
$\theta^{(k+1)}_k$ as 
\[
\theta^{(k+1)}_k = \frac{\alpha_k\gamma_k}{\alpha_k \gamma_k+\gamma_{k+1}}\left(1-\frac{\mu}{n}\right) (1-n\alpha_{k-1}) + \left(1-\frac{\mu}{n}\right)n(\alpha_{k-1}-\alpha_k) 
+ \frac{\mu}{n} (1-n\alpha_k).
\]
Together with $\mu \le 1$, $0 \le \alpha_k \le 1/n$, $\gamma_k \ge 0$ and 
$\alpha_{k-1} \ge \alpha_k$, this implies that $\theta^{(k+1)}_k \ge 0$. 
Therefore, $x^{(k+1)}$ is a convex combination of $z^{(0)},\ldots,z^{(k+1)}$ 
with the coefficients given in \eqref{eqn:combine-coeff}. 
\end{proof}

In the following lemma, we construct the sequence $\{\hat\Psi_k\}_{k=0}^\infty$
and prove a recursive inequality.

\begin{lemma}\label{lem:mono-combine}
Let $\hat \Psi_k$ denotes the convex combination of 
$\Psi(z^{(0)}),\ldots,\Psi(z^{(k)})$ using the same coefficients given in
Lemma~\ref{lem:combination}, i.e., 
\[
    \hat\Psi_k = \sum_{l=0}^k \theta^{(k)}_l \Psi(z^{(l)}).
\]
Then for all $k\geq 0$, we have $\Psi(x^{(k)}) \leq \hat\Psi_k$ and
\begin{equation}\label{eqn:mono-combine}
\E_{i_k}[\hat\Psi_{k+1}] ~\leq~ \alpha_k\Psi(\tz^{(k+1)}) + (1-\alpha_k)\hat\Psi_k.
\end{equation}
\end{lemma}
\begin{proof} 
The first result $\Psi(x^{(k)})\leq\hat\Psi_k$ follows directly from convexity
of~$\Psi$.
We now prove~\eqref{eqn:mono-combine}. First we deal with the case $k=0$. 
Using \eqref{eqn:z-update}, \eqref{eqn:combine-coeff}, and the facts $y^{(0)}=z^{(0)}$ and $\hat\Psi_0=\Psi(x^{(0)})$, we get
\begin{eqnarray*}
    \E_{i_0}[\hat\Psi_1] 
&=& \E_{i_0}\left[ n\alpha_0\Psi(z^{(1)}) + (1-n\alpha_0)\Psi(z^{(0)})\right] \\
&=& \E_{i_0}\Bigl[ n\alpha_0\Bigl(\Psi_{i_0}(\tz^{(1)}_{i_0})
  +\textstyle\sum_{j\neq i_0}\Psi_j(z_j^0)\Bigr)\Bigr] + (1-n\alpha_0)\Psi(z^{(0)})\\
&=& \alpha_0 \Psi(\tz^{(1)}) + (n-1)\alpha_0\Psi(z^{(0)}) + (1-n\alpha_0)\Psi(x^{(0)})\\
&=& \alpha_0 \Psi(\tz^{(1)}) + (1-\alpha_0)\Psi(x^{(0)}) \\ 
&=& \alpha_0 \Psi(\tz^{(1)}) + (1-\alpha_0)\hat\Psi_0 .
\end{eqnarray*}

For $k\geq 1$, we use \eqref{eqn:z-update} and the definition of~$\beta_k$
in~\eqref{gammak} to obtain that 
\beqa
\E_{i_k}\left[\Psi(z^{(k+1)})\right] 
&=&  \E_{i_k}\biggl[\Psi_{i_k}(z^{(k+1)}_{i_k})
+\sum_{j\neq i_k}\Psi_j(z^{(k+1)}_j)\biggr] \nn \\
&=& \frac1n\Psi(\tz^{(k+1)}) + \left(1-\frac1n\right)\Psi\left( \frac{(1-\alpha_k)\gamma_k}{\gamma_{k+1}} z^{(k)} +\frac{\alpha_k \mu}{\gamma_{k+1}}  y^{(k)} \right). \label{exp-ik1}  
\eeqa
Using \eqref{gammak} and \eqref{yk}, one can observe that 
\beqas
\frac{(1-\alpha_k)\gamma_k}{\gamma_{k+1}}z^{(k)} +\frac{\alpha_k \mu}{\gamma_{k+1}}  y^{(k)} &=& 
\frac{(1-\alpha_k)\gamma_k}{\gamma_{k+1}}z^{(k)} +\frac{\alpha_k \mu}{\gamma_{k+1} (\alpha_k \gamma_k+\gamma_{k+1})} \left(\alpha_k \gamma_k z^{(k)} + \gamma_{k+1} x^{(k)}\right) \nn \\ [6pt]
&=& \left(1-\frac{\alpha_k \mu}{\alpha_k \gamma_k+\gamma_{k+1}}\right) z^{(k)}+ \frac{\alpha_k \mu}{\alpha_k \gamma_k+\gamma_{k+1}} x^{(k)}. 
\eeqas
It follows from the above equation and convexity of $\Psi$ that 
\[
\Psi\left( \frac{(1-\alpha_k)\gamma_k}{\gamma_{k+1}} z^{(k)} +\frac{\alpha_k \mu}{\gamma_{k+1}}  y^{(k)} \right) \le \left(1-\frac{\alpha_k \mu}{\alpha_k \gamma_k+\gamma_{k+1}}\right) \Psi(z^{(k)}) + 
\frac{\alpha_k \mu}{\alpha_k \gamma_k+\gamma_{k+1}} \Psi(x^{(k)}),
\]
which together with \eqref{exp-ik1} yields 
\beq \label{exp-ik2}
\E_{i_k}\left[\Psi(z^{(k+1)})\right] \le 
\frac1n\Psi(\tz^{(k+1)}) + \left(1-\frac1n\right) \left[\left(1-\frac{\alpha_k \mu}{\alpha_k \gamma_k+\gamma_{k+1}}\right) \Psi(z^{(k)}) + 
\frac{\alpha_k \mu}{\alpha_k \gamma_k+\gamma_{k+1}} \Psi(x^{(k)})\right]. 
\eeq 
In addition, from the definition of $\hat\Psi_k$ and
$\theta^{(k)}_k=n\alpha_{k-1}$, we have
\beq \label{hpsik}
\sum_{l=0}^{k-1}\theta^{(k)}_l \Psi(z^{(l)}) = \hat\Psi_k - n\alpha_{k-1} \Psi(z^{(k)}).
\eeq

Next, using the definition of $\hat\Psi_k$ and \eqref{eqn:combine-coeff},
we obtain 
\begin{eqnarray}
\E_{i_k}[\hat\Psi_{k+1}]
&=& n\alpha_k \E_{i_k}\left[\Psi(z^{(k+1)}) \right]+  \left[\left(1-\frac{\mu}{n}\right) \frac{\alpha_k\gamma_k+n\alpha_{k-1}\gamma_{k+1}}
{\alpha_k \gamma_k+\gamma_{k+1}}-\frac{(1-\alpha_k)\gamma_k}{n\alpha_k}\right] \Psi(z^{(k)})  \nonumber \\
&& +  \left(1-\frac{\mu}{n}\right) \frac{\gamma_{k+1}}{\alpha_k \gamma_k+\gamma_{k+1}} \sum_{l=0}^{k-1}\theta^{(k)}_l \Psi(z^{(l)}) . 
\label{exp-ik3} 
\end{eqnarray}
Plugging~\eqref{exp-ik2} and~\eqref{hpsik} into~\eqref{exp-ik3} yields
\begin{eqnarray}
\E_{i_k}[\hat\Psi_{k+1}]
&\le& \alpha_k \Psi(\tz^{(k+1)}) + (n-1) \alpha_k \left[\left(1-\frac{\alpha_k \mu}{\alpha_k 
 \gamma_k+\gamma_{k+1}}\right) \Psi(z^{(k)}) + 
\frac{\alpha_k \mu}{\alpha_k \gamma_k+\gamma_{k+1}} \Psi(x^{(k)})\right] 
\nn \\[1ex]
&& + \left[\left(1-\frac{\mu}{n}\right) \frac{\alpha_k\gamma_k+n\alpha_{k-1}\gamma_{k+1}}
{\alpha_k \gamma_k+\gamma_{k+1}}-\frac{(1-\alpha_k)\gamma_k}{n\alpha_k}\right] \Psi(z^{(k)}) 
\label{exp-ik4}\\[1ex]
&& + \left(1-\frac{\mu}{n}\right) \frac{\gamma_{k+1}}{\alpha_k \gamma_k+\gamma_{k+1}} \left(
\hat\Psi_k - n\alpha_{k-1} \Psi(z^{(k)})\right) \nn \\[1ex]
& \le & \alpha_k \Psi(\tz^{(k+1)}) 
+ \underbrace{\frac{ (n-1) \alpha^2_k\mu+\left(1-\frac{\mu}{n}\right)\gamma_{k+1}}{\alpha_k \gamma_k+\gamma_{k+1}}}_{\Gamma} \hat \Psi_k 
\label{exp-ik5} \\
&&+ \underbrace{\left[(n-1) \alpha_k \left(1-\frac{\alpha_k \mu}{\alpha_k 
 \gamma_k+\gamma_{k+1}}\right) + \left(1-\frac{\mu}{n}\right) \frac{\alpha_k\gamma_k}
{\alpha_k \gamma_k+\gamma_{k+1}}-\frac{(1-\alpha_k)\gamma_k}{n\alpha_k}\right]}_{\Delta}\Psi(z^{(k)}), \nn 
\end{eqnarray}
where the second inequality is due to $\Psi(x^{(k)}) \le \hat \Psi_k$. 
Notice that the right hand side of \eqref{exp-ik2} is an affine combination 
of $\Psi(\tz^{(k+1)})$, $\Psi(z^{(k)})$ and $\Psi(x^{(k)})$, 
and the right hand side of \eqref{exp-ik3} is an affine combination of 
$\Psi(z^{(0)}), \ldots, \Psi(z^{(k+1)})$. 
In addition, all operations in~\eqref{exp-ik4} and~\eqref{exp-ik5} preserves 
the affine combination property.   
Using these facts, one can observe that the 
right hand side of \eqref{exp-ik5} is also an affine combination of 
$\Psi(\tz^{(k+1)})$, $\Psi(z^{(k)})$ and $\hat\Psi_k$, namely, 
$\alpha_k+\Delta+\Gamma=1$, where $\Delta$ and $\Gamma$ are defined above. 

We next show that $\Gamma=1-\alpha_k$ and $\Delta=0$. 
Indeed, notice that from~\eqref{gammak} we have
\beq\label{gamma-alpha}
\alpha_k \gamma_k+\gamma_{k+1} = \alpha_k \mu + \gamma_k.
\eeq
Using this relation, 
$\gamma_{k+1}=n^2\alpha^2_k$ (Lemma \ref{alphak-prop} (iv)), 
and the definition of $\Gamma$ in~\eqref{exp-ik5}, we get
\beqas
\Gamma &=& \frac{ (n-1) \alpha^2_k\mu+\left(1-\frac{\mu}{n}\right)\gamma_{k+1}}{\alpha_k \gamma_k+\gamma_{k+1}} 
~=~ \frac{ (n-1) \alpha^2_k\mu+\gamma_{k+1}-\frac{\mu}{n}\gamma_{k+1}}{\alpha_k \gamma_k+\gamma_{k+1}} \\ [6pt]
&=& \frac{ (n-1) \alpha^2_k\mu+\gamma_{k+1}-\frac{\mu}{n}(n^2\alpha^2_k)}{\alpha_k \gamma_k+\gamma_{k+1}} 
~=~  \frac{\gamma_{k+1}-\alpha^2_k\mu}{\alpha_k \gamma_k+\gamma_{k+1}} \\
&=& 1 - \alpha_k \left(\frac{\alpha_k\mu+\gamma_k}{\alpha_k \gamma_k+\gamma_{k+1}}\right) 
~=~ 1-\alpha_k, 
\eeqas
where the last equalities is due to \eqref{gamma-alpha}.  
Finally, $\Delta=0$ follows from $\Gamma=1-\alpha_k$ and 
$\alpha_k+\Delta+\Gamma=1$.
These together with the inequality~\eqref{exp-ik5} yield the desired result. 
\end{proof}

\subsection{Proof of Theorem~\ref{thm:apcg-rate}}
\label{sec:prove-thm1}

We are now ready to present a proof for Theorem~\ref{thm:apcg-rate}.
We note that the proof in this subsection can also be recast into 
the framework of randomized estimate sequence developed 
in \cite{LuXiao13analysis,LeeSidford13},
but here we give a straightforward proof without using that machinery.

Dividing both sides of \eqref{alphak} by $n\alpha_k$ gives  
\beq \label{nalphak}
n\alpha_k = \frac{(1-\alpha_k)\gamma_k}{n\alpha_k}+\frac{\mu}{n}.
\eeq
Observe from \eqref{yk} that 
\beq \label{z-y}
z^{(k)}-y^{(k)} = -\frac{\gamma_{k+1}}{\alpha_k\gamma_k}\left(x^{(k)}-y^{(k)}\right) .
\eeq
It follow from~\eqref{xk} and~\eqref{nalphak} that 
\beqas
x^{(k+1)} - y^{(k)} &=& n\alpha_k z^{(k+1)} - \frac{(1-\alpha_k)\gamma_k}{n\alpha_k}  z^{(k)} - \frac{\mu}{n} y^{(k)} \\ 
&=& n\alpha_k z^{(k+1)}  - \frac{(1-\alpha_k)\gamma_k}{n\alpha_k}(z^{(k)}-y^{(k)}) -  \left(\frac{(1-\alpha_k)\gamma_k}{n\alpha_k}+\frac{\mu}{n} \right)y^{(k)},
\eeqas
which together with \eqref{nalphak}, \eqref{z-y} and $\gamma_{k+1}=n^2\alpha^2_k$ 
(Lemma \ref{alphak-prop} (iv)) gives 
\beqas
x^{(k+1)} - y^{(k)} &=& n\alpha_k z^{(k+1)} + \frac{(1-\alpha_k)\gamma_{k+1}}{n\alpha^2_k}\left(x^{(k)}-y^{(k)}\right) - n\alpha_k y^{(k)} \\
&=& n \alpha_k z^{(k+1)} + n (1-\alpha_k)(x^{(k)}-y^{(k)}) - n\alpha_k y^{(k)} \\
&=&n \left[\alpha_k(z^{(k+1)}-y^{(k)})+(1-\alpha_k)(x^{(k)}-y^{(k)})\right].
\eeqas
Using this relation, \eqref{eqn:x-update} and Assumption~\ref{asmp:coord-smooth}, we have
\begin{eqnarray*}
f(x^{(k+1)})
&\leq& f(y^{(k)})+\left\langle\nabla_{i_k} f(y^{(k)}), ~x^{(k+1)}_{i_k}-y^{(k)}_{i_k}\right\rangle 
+~\frac{L_{i_k}}{2}\left\|x^{(k+1)}_{i_k}-y^{(k)}_{i_k}\right\|_2^2 \\ 
&=& f(y^{(k)})+ n\left\langle\nabla_{i_k} f(y^{(k)}),~\left[\alpha_k(z^{(k+1)}-y^{(k)})+(1-\alpha_k)(x^{(k)}-y^{(k)})\right]_{i_k}\right\rangle\\
&& +~\frac{n^2 L_{i_k}}{2}\left\|\left[\alpha_k(z^{(k+1)}-y^{(k)})+(1-\alpha_k)(x^{(k)}-y^{(k)})\right]_{i_k}\right\|_2^2 \\
&=& (1-\alpha_k)\left[f(y^{(k)})+ n\left\langle\nabla_{i_k} f(y^{(k)}),(x^{(k)}_{i_k}-y^{(k)}_{i_k})\right\rangle\right] \\
&& + \alpha_k\left[f(y^{(k)})+ n\left\langle\nabla_{i_k} f(y^{(k)}),(z^{(k+1)}_{i_k}-y^{(k)}_{i_k})\right\rangle\right] \\
&& +~\frac{n^2 L_{i_k}}{2}\left\|\left[\alpha_k(z^{(k+1)}-y^{(k)})+(1-\alpha_k)(x^{(k)}-y^{(k)})\right]_{i_k}\right\|_2^2.
\end{eqnarray*}
Taking expectation on both sides of the above inequality with respect to~$i_k$,
and noticing that $z^{(k+1)}_{i_k} = \tz^{(k+1)}_{i_k}$, we get
\begin{eqnarray}
\E_{i_k}\left[f(x^{(k+1)})\right]
&\leq& (1-\alpha_k)\left[f(y^{(k)})+ \left\langle\nabla f(y^{(k)}),(x^{(k)}-y^{(k)})\right\rangle\right] \nn \\
&& + \alpha_k\left[f(y^{(k)})+ \left\langle\nabla f(y^{(k)}),(\tz^{(k+1)}-y^{(k)})\right\rangle\right] \nn \\
&& +~\frac{n}{2}\left\|\alpha_k(\tz^{(k+1)}-y^{(k)})+(1-\alpha_k)(x^{(k)}-y^{(k)})\right\|_L^2 \nn \\
&\leq& (1-\alpha_k) f(x^{(k)}) + \alpha_k\left[f(y^{(k)})+ \left\langle\nabla f(y^{(k)}),(\tz^{(k+1)}-y^{(k)})\right\rangle \right] \nn \\
&& +~\frac{n}{2}\left\|\alpha_k(\tz^{(k+1)}-y^{(k)})+(1-\alpha_k)(x^{(k)}-y^{(k)})\right\|_L^2, \label{exp-f-ik}
\end{eqnarray}
where the second inequality follows from convexity of $f$.  

In addition, by \eqref{gammak}, \eqref{z-y} and $\gamma_{k+1}=n^2\alpha^2_k$ (Lemma \ref{alphak-prop} (iv)),  we have
\beqa
\frac{n}{2}\left\|\alpha_k(\tz^{(k+1)}-y^{(k)})+(1-\alpha_k)(x^{(k)}-y^{(k)})\right\|_L^2 
&=& \frac{n}{2}\left\|\alpha_k(\tz^{(k+1)}-y^{(k)})-\frac{\alpha_k(1-\alpha_k)\gamma_k}{\gamma_{k+1}}(z^{(k)}-y^{(k)})\right\|_L^2 \nn \\ 
&=& \frac{n\alpha^2_k}{2}\left\|\tz^{(k+1)}-y^{(k)}-\frac{(1-\alpha_k)\gamma_k}{\gamma_{k+1}}(z^{(k)}-y^{(k)})\right\|_L^2 \nn \\ 
&=& \frac{\gamma_{k+1}}{2n}\left\|\tz^{(k+1)}-\frac{(1-\alpha_k)\gamma_k}{\gamma_{k+1}}z^{(k)}-\frac{\alpha_k\mu}{\gamma_{k+1}}y^{(k)}\right\|_L^2,\label{eqn:B-first}
\eeqa
where the first equality used~\eqref{z-y}, 
the third one is due to~\eqref{alphak} and~\eqref{gammak}, 
and $\gamma_{k+1}=n^2\alpha^2_k$. 
This equation together with~\eqref{exp-f-ik} yields 
\begin{eqnarray*}
\E_{i_k}\left[f(x^{(k+1)})\right] 
&\leq& (1-\alpha_k) f(x^{(k)}) 
 + \alpha_k\Bigg[f(y^{(k)})+ \left\langle\nabla f(y^{(k)}),\tz^{(k+1)}-y^{(k)}\right\rangle \\
&& \hspace{4cm}
+ \frac{\gamma_{k+1}}{2n\alpha_k}\left\|\tz^{(k+1)}-\frac{(1-\alpha_k)\gamma_k}{\gamma_{k+1}}z^{(k)}-\frac{\alpha_k\mu}{\gamma_{k+1}}y^{(k)}\right\|_L^2 
\Bigg].
\end{eqnarray*}
Using Lemma~\ref{lem:mono-combine}, we have
\begin{eqnarray*}
\E_{i_k}\left[f(x^{(k+1)}) + \hat\Psi_{k+1} \right]
\leq \E_{i_k}[f(x^{(k+1)})]+\alpha_k\Psi(\tz^{(k+1)}) + (1-\alpha_k) \hat\Psi_k.
\end{eqnarray*}
Combining the above two inequalities, one can obtain that
\begin{eqnarray}
\E_{i_k}\left[f(x^{(k+1)}) + \hat\Psi_{k+1} \right]
&\leq& (1-\alpha_k) \left( f(x^{(k)}) + \hat\Psi_k\right) 
+ \alpha_k V(\tz^{(k+1)}),
\label{F-uppbd1}
\end{eqnarray}
where
\[
V(x) = f(y^{(k)})+ \left\langle\nabla f(y^{(k)}),x-y^{(k)}\right\rangle  
+ \frac{\gamma_{k+1}}{2n\alpha_k}\left\|x-\frac{(1-\alpha_k)\gamma_k}{\gamma_{k+1}}z^{(k)}-\frac{\alpha_k\mu}{\gamma_{k+1}}y^{(k)}\right\|_L^2 +\Psi(x).
\]
Comparing with the definition of $\tz^{(k+1)}$ in~\eqref{eqn:full-z-update},
we see that 
\beq \label{tz}
\tz^{(k+1)} = \argmin_{x \in \Re^N} V(x).
\eeq
Notice that $V$ has convexity parameter 
$\frac{\gamma_{k+1}}{n\alpha_k}=n\alpha_k$ 
with respect to $\|\cdot\|_L$. 
By the optimality condition of \eqref{tz}, 
we have that for any $x^\star \in X^*$,
\[
V(x^\star) \ge V(\tz^{(k+1)}) + \frac{\gamma_{k+1}}{2n\alpha_k}\|x^\star-\tz^{(k+1)}\|^2_L.
\]
Using the above inequality and the definition of $V$, 
we obtain 
\begin{eqnarray*}
V(\tz^{(k+1)}) &\le & V(x^\star) - \frac{\gamma_{k+1}}{2n\alpha_k}\|x^\star-\tz^{(k+1)}\|^2_L \\
&=& f(y^{(k)})+ \left\langle\nabla f(y^{(k)}),x^\star-y^{(k)}\right\rangle 
+ \frac{\gamma_{k+1}}{2n\alpha_k}\left\|x^\star-\frac{(1-\alpha_k)\gamma_k}{\gamma_{k+1}}z^{(k)}-\frac{\alpha_k\mu}{\gamma_{k+1}}y^{(k)}\right\|_L^2 \\
&& +\Psi(x^\star) 
- \frac{\gamma_{k+1}}{2n\alpha_k}\|x^\star-\tz^{(k+1)}\|^2_L .
\end{eqnarray*}
Now using the assumption that~$f$ has convexity parameter~$\mu$ 
with respect to $\|\cdot\|_L$, we have
\begin{eqnarray*}
V(\tz^{(k+1)}) &\le& f(x^\star) - \frac{\mu}{2} \|x^\star-y^{(k)}\|^2_L + \frac{\gamma_{k+1}}{2n\alpha_k}\left\|x^\star-\frac{(1-\alpha_k)\gamma_k}{\gamma_{k+1}}z^{(k)}-\frac{\alpha_k\mu}{\gamma_{k+1}}y^{(k)}\right\|_L^2 +\Psi(x^\star) \\
&& - \frac{\gamma_{k+1}}{2n\alpha_k}\|x^\star-\tz^{(k+1)}\|^2_L .
\end{eqnarray*}
Combining this inequality with \eqref{F-uppbd1}, one see that
\begin{eqnarray}
\E_{i_k}\left[f(x^{(k+1)}) + \hat\Psi_{k+1} \right]
&\leq& (1-\alpha_k) \left( f(x^{(k)}) 
+ \hat\Psi_k\right) + \alpha_k F^\star - \frac{\alpha_k\mu}{2} \|x^\star-y^{(k)}\|^2_L 
\label{F-uppbd} \\
&& + \frac{\gamma_{k+1}}{2n}\left\|x^\star-\frac{(1-\alpha_k)\gamma_k}{\gamma_{k+1}}z^{(k)}-\frac{\alpha_k\mu}{\gamma_{k+1}}y^{(k)}\right\|_L^2 - \frac{\gamma_{k+1}}{2n}\|x^\star-\tz^{(k+1)}\|^2_L.  \nn
\end{eqnarray}
In addition, it follows from \eqref{gammak} and convexity of $\|\cdot\|^2_L$ that  
\beq \label{Lnorm-ineq}
\left\|x^\star-\frac{(1-\alpha_k)\gamma_k}{\gamma_{k+1}}z^{(k)}-\frac{\alpha_k\mu}{\gamma_{k+1}}y^{(k)}\right\|_L^2 \le \frac{(1-\alpha_k)\gamma_k}{\gamma_{k+1}} 
\|x^\star-z^{(k)}\|^2_L+\frac{\alpha_k\mu}{\gamma_{k+1}}\|x^\star-y^{(k)}\|_L^2.
\eeq
Using this relation and \eqref{eqn:z-update}, we observe that 
\begin{eqnarray*}
\E_{i_k}\left[\frac{\gamma_{k+1}}{2}\|x^\star-z^{(k+1)}\|^2_L\right] &=& \frac{\gamma_{k+1}}{2} 
\left[\frac{n-1}{n}\left\|x^\star-\frac{(1-\alpha_k)\gamma_k}{\gamma_{k+1}}z^{(k)}-\frac{\alpha_k\mu}{\gamma_{k+1}}y^{(k)}\right\|_L^2+\frac{1}{n}\|x^\star-\tz^{(k+1)}\|^2_L\right] \\
&=& \frac{\gamma_{k+1}(n-1)}{2n}\left\|x^\star-\frac{(1-\alpha_k)\gamma_k}{\gamma_{k+1}}z^{(k)}-\frac{\alpha_k\mu}{\gamma_{k+1}}y^{(k)}\right\|_L^2+\frac{\gamma_{k+1}}{2n}\|x^\star-\tz^{(k+1)}\|^2_L \\
&=&\frac{\gamma_{k+1}}{2}\left\|x^\star-\frac{(1-\alpha_k)\gamma_k}{\gamma_{k+1}}z^{(k)}-\frac{\alpha_k\mu}{\gamma_{k+1}}y^{(k)}\right\|_L^2 \\
&& - \frac{\gamma_{k+1}}{2n}\left\|x^\star-\frac{(1-\alpha_k)\gamma_k}{\gamma_{k+1}}z^{(k)}-\frac{\alpha_k\mu}{\gamma_{k+1}}y^{(k)}\right\|_L^2 
+\frac{\gamma_{k+1}}{2n}\|x^\star-\tz^{(k+1)}\|^2_L \\
&\le& \frac{(1-\alpha_k)\gamma_k}{2}\|x^\star-z^{(k)}\|_L^2+\frac{\alpha_k\mu}{2}\|x^\star-y^{(k)}\|_L^2 \\
&& - \frac{\gamma_{k+1}}{2n}\left\|x^\star-\frac{(1-\alpha_k)\gamma_k}{\gamma_{k+1}}z^{(k)}-\frac{\alpha_k\mu}{\gamma_{k+1}}y^{(k)}\right\|_L^2 
+\frac{\gamma_{k+1}}{2n}\|x^\star-\tz^{(k+1)}\|^2_L,
\end{eqnarray*}
where the inequality follows from \eqref{Lnorm-ineq}. Summing up this inequality and \eqref{F-uppbd} 
gives
\[
\E_{i_k}\left[f(x^{(k+1)}) + \hat\Psi_{k+1}+\frac{\gamma_{k+1}}{2}\|x^\star-z^{(k+1)}\|^2_L\right] 
\le  (1-\alpha_k) \left( f(x^{(k)}) 
+ \hat\Psi_k+\frac{\gamma_k}{2}\|x^\star-z^{(k)}\|^2_L\right) + \alpha_k F^\star.
\]
Taking expectation on both sides with respect to $\xi_{k-1}$ yields
\[
\E_{\xi_k}\left[f(x^{(k+1)}) + \hat\Psi_{k+1}-F^\star+\frac{\gamma_{k+1}}{2}\|x^\star-z^{(k+1)}\|^2_L\right] 
\le  (1-\alpha_k) \E_{\xi_{k-1}}\left[ f(x^{(k)}) 
+ \hat\Psi_k-F^\star+\frac{\gamma_k}{2}\|x^\star-z^{(k)}\|^2_L\right],
\]
which together with $\hat\Psi_0=\Psi(x^{(0)})$, $z^{(0)}=x^{(0)}$ and $\lambda_k=\Pi^{k-1}_{i=0}(1-\alpha_i)$ 
gives 
\[
\E_{\xi_{k-1}}\left[ f(x^{(k)}) 
+ \hat\Psi_k-F^\star+\frac{\gamma_k}{2}\|x^\star-z^{(k)}\|^2_L\right] \le \lambda_k\left[F(x^{(0)}) -F^\star+\frac{\gamma_0}{2}\|x^\star-x^{(0)}\|^2_L\right].
\]
The conclusion of Theorem~\ref{thm:apcg-rate} immediately follows from 
$F(x^{(k)})\leq f(x^{(k)}) + \hat\Psi_k$,
Lemma \ref{alphak-prop} (v), the arbitrariness of $x^\star$ and the definition of $R_0$.

\section{Efficient implementation}
\label{sec:effi-impl}

The APCG methods we presented in Section~\ref{sec:apcg} all need to perform 
full-dimensional vector operations at each iteration.
In particular, $y^{(k)}$ is updated as a convex combination of 
$x^{(k)}$ and $z^{(k)}$, and this can be very costly since in general 
they are dense vectors in $\reals^N$.
Moreover, in the strongly convex case 
(Algorithms~\ref{alg:apcg} and~\ref{alg:apcg-mu}), 
all blocks of $z^{(k+1)}$ also need to be updated at each iteration, 
although only the $i_k$th block needs to compute the partial gradient 
and perform an proximal mapping of $\Psi_{i_k}$.
These full-dimensional vector updates cost $O(N)$ operations per iteration
and may cause the overall computational cost of APCG to be comparable or 
even higher than the full gradient methods 
(see discussions in \cite{Nesterov12rcdm}).

In order to avoid full-dimensional vector operations, 
Lee and Sidford \cite{LeeSidford13} proposed a change of variables scheme
for accelerated coordinated gradient methods for unconstrained smooth 
minimization. 
Fercoq and Richt\'arik \cite{FercoqRichtarik13} devised a similar scheme
for efficient implementation in the non-strongly convex case ($\mu=0$) 
for composite minimization.
Here we show that full vector operations can also be avoided in the 
strongly convex case for minimizing composite functions.
For simplicity, we only present an efficient implementation of the
simplified APCG method with $\mu>0$ (Algorithm~\ref{alg:apcg-mu}),
which is given as Algorithm~\ref{alg:apcg-mu-effi}.

\begin{algorithm}[t]
\caption{Efficient implementation of APCG with $\gamma_0=\mu>0$}
\label{alg:apcg-mu-effi}
\textbf{input:} $x^{(0)}\in\dom(\Psi)$ and convexity parameter $\mu>0$.\\[0.5ex]
\textbf{initialize:} 
set $\alpha=\frac{\sqrt{\mu}}{n}$ and $\rho=\frac{1-\alpha}{1+\alpha}$,
and initialize $u^{(0)}=0$ and $v^{(0)}=x^{(0)}$.\\[1ex]
\textbf{iterate:} repeat for $k=0,1,2,\ldots$ 
\begin{enumerate} \itemsep 0pt 
\item Choose $i_k\in\{1,\ldots,n\}$ uniformly at random and compute
\[
    h^{(k)}_{i_k} = \argmin_{h\in\reals^{N_{i_k}}} 
   \left\{ \frac{n\alpha L_{i_k}}{2} \|h \|_2^2 
   + \left\langle \nabla_{\!i_k} f(\rho^{k+1}u^{(k)}\!+\!v^{(k)}),\, h\right\rangle 
    + \Psi_{i_k}\Bigl(-\rho^{k+1}u^{(k)}_{i_k}\!+\!v^{(k)}_{i_k} \!+\! h\Bigr)\right\}.
\vspace{-1ex}
\]
\item Let $u^{(k+1)} = u^{(k)}$ and $v^{(k+1)} = v^{(k)}$, and update
\begin{equation}\label{ukvk}
u^{(k+1)}_{i_k} = u^{(k)}_{i_k} - \frac{1-n\alpha}{2\rho^{k+1}} h^{(k)}_{i_k}, \qquad
v^{(k+1)}_{i_k} = v^{(k)}_{i_k} + \frac{1+n\alpha}{2} h^{(k)}_{i_k}.
\vspace{-1ex}
\end{equation}
\end{enumerate}
\textbf{output:} $x^{(k+1)} = \rho^{k+1} u^{(k+1)} + v^{(k+1)}$
\end{algorithm}


\begin{proposition}\label{prop:equiv}
    The iterates of Algorithm~\ref{alg:apcg-mu} and Algorithm~\ref{alg:apcg-mu-effi} satisfy the following relationships:
\begin{eqnarray} 
    x^{(k)} &=& \rho^k u^{(k)} + v^{(k)}, \nonumber \\
    y^{(k)} &=& \rho^{k+1} u^{(k)} + v^{(k)}, \label{eqn:equiv} \\
    z^{(k)} &=& -\rho^k u^{(k)} + v^{(k)} , \nonumber
\end{eqnarray}
for all $k\geq 0$.
That is, these two algorithms are equivalent.
\end{proposition}

\begin{proof}
We prove by induction.
Notice that Algorithm~\ref{alg:apcg-mu} is initialized with $z^{(0)}=x^{(0)}$, 
and its first step implies $y^{(0)}=\frac{x^{(0)}+\alpha z^{(0)}}{1+\alpha}=x^{(0)}$;
Algorithm~\ref{alg:apcg-mu-effi} is initialized with $u^{(0)}=0$ and $v^{(0)}=x^{(0)}$.
Therefore we have
\[
    x^{(0)} = \rho^0 u^{(0)} + v^{(0)}, \qquad
    y^{(0)} = \rho^1 u^{(0)} + v^{(0)}, \qquad
    z^{(0)} = - \rho^0 u^{(0)} + v^{(0)},
\]
which means that~\eqref{eqn:equiv} holds for $k=0$.
Now suppose that it holds for some $k\geq0$, then
\begin{eqnarray}
    (1-\alpha)z^{(k)} + \alpha y^{(k)} 
&=& (1-\alpha)\left(-\rho^k u^{(k)} + v^{(k)}\right) + \alpha\left(\rho^{k+1} u^{(k)} + v^{(k)}\right) \nonumber \\
&=& -\rho^k \left((1-\alpha) - \alpha\rho \right) u^{(k)} + (1-\alpha)v^{(k)} + \alpha v^{(k)} \nonumber \\
&=& -\rho^{k+1} u^{(k)} + v^{(k)} . \label{eqn:combine-zkyk}
\end{eqnarray}
So $h^{(k)}_{i_k}$ in Algorithm~\ref{alg:apcg-mu-effi} can be written as
\[
    h^{(k)}_{i_k} = \argmin_{h\in\reals^{N_{i_k}}} 
    \left\{ \frac{n\alpha L_{i_k}}{2} \|h \|_2^2 
    + \langle \nabla_{i_k} f(y^{(k)}), h\rangle 
    + \Psi_{i_k}\left((1-\alpha)z^{(k)}_{i_k} +\alpha y^{(k)}_{i_k} + h\right) \right\} .
\]
Comparing with~\eqref{eqn:full-z-update}, and using $\beta_k=\alpha$, we obtain
\[
    h^{(k)}_{i_k} = \tilde z^{(k+1)}_{i_k} - \bigl( (1-\alpha)z^{(k)}_{i_k} +\alpha y^{(k)}_{i_k} \bigr).
\]
In terms of the full dimensional vectors, 
using~\eqref{eqn:z-update} and~\eqref{eqn:combine-zkyk}, we have
\begin{eqnarray*}
    z^{(k+1)} 
&=& (1-\alpha)z^{(k)} + \alpha y^{(k)} + U_{i_k}h^{(k)}_{i_k} \\
&=& -\rho^{k+1}u^{(k)} + v^{(k)} + U_{i_k}h^{(k)}_{i_k} \\
&=& -\rho^{k+1}u^{(k)} + v^{(k)} + \frac{1-n\alpha}{2} U_{i_k}h^{(k)}_{i_k} 
    + \frac{1+n\alpha}{2} U_{i_k}h^{(k)}_{i_k} \\
&=& -\rho^{k+1}\left(u^{(k)}-\frac{1-n\alpha}{2\rho^{k+1}} U_{i_k}h^{(k)}_{i_k}\right)
    +\left( v^{(k)} + \frac{1+n\alpha}{2} U_{i_k} h^{(k)}_{i_k} \right) \\
&=& -\rho^{k+1} u^{(k+1)} + v^{(k+1)} .
\end{eqnarray*}
Using Step~3 of Algorithm~\ref{alg:apcg-mu}, we get
\begin{eqnarray*}
    x^{(k+1)}
&=& y^{(k)} + n\alpha(z^{(k+1)}-z^{(k)})+n\alpha^2(z^{(k)}-y^{(k)}) \\
&=& y^{(k)} + n\alpha\left(z^{(k+1)} - \bigl((1-\alpha)z^{(k)}+\alpha y^{(k)}\bigr)\right)\\
&=& y^{(k)} + n\alpha U_{i_k}h^{(k)}_{i_k} ,
\end{eqnarray*}
where the last step used~\eqref{eqn:z-update}.
Now using the induction hypothesis $y^{(k)}=\rho^{k+1}u^{(k)}+v^{(k)}$, we have
\begin{eqnarray*}
    x^{(k+1)}
&=& \rho^{k+1} u^{(k)} + v^{(k)} +  \frac{1-n\alpha}{2} U_{i_k}h^{(k)}_{i_k} 
    + \frac{1+n\alpha}{2} U_{i_k}h^{(k)}_{i_k} \\
&=& \rho^{k+1} \left(u^{(k)}-\frac{1-n\alpha}{2\rho^{k+1}} U_{i_k}h^{(k)}_{i_k}\right)
    +\left( v^{(k)} + \frac{1+n\alpha}{2} U_{i_k} h^{(k)}_{i_k} \right) \\
&=& \rho^{k+1} u^{(k+1)} + v^{(k+1)} .
\end{eqnarray*}
Finally,
\begin{eqnarray*}
    y^{(k+1)}
&=&  \frac{1}{1+\alpha}\left(x^{(k+1)} + \alpha z^{(k+1)}\right)\\
&=&  \frac{1}{1+\alpha}\left(\rho^{k+1}u^{(k+1)}+v^{(k+1)}\right)
    +\frac{\alpha}{1+\alpha}\left(-\rho^{k+1}u^{(k+1)}+v^{(k+1)}\right)\\
&=&  \frac{1-\alpha}{1+\alpha}\rho^{k+1}u^{(k+1)}
    +\frac{1+\alpha}{1+\alpha}v^{(k+1)} \\
&=& \rho^{k+2}u^{(k+1)} + v^{(k+1)} .
\end{eqnarray*}
We just showed that~\eqref{eqn:equiv} also holds for $k+1$.
This finishes the induction.
\end{proof}

We note that in Algorithm~\ref{alg:apcg-mu-effi}, 
only a single block coordinates of the vectors $u^{(k)}$ and $v^{(k)}$ 
are updated at each iteration, which cost $O(N_{i_k})$.
However, computing the partial gradient 
$\nabla_{i_k}f(\rho^{k+1}u^{(k)}+v^{(k)})$ may still cost $O(N)$ in general.
In Section~\ref{sec:erm-impl}, 
we show how to further exploit problem structure in 
regularized empirical risk minimization 
to completely avoid full-dimensional vector operations.

\section{Application to regularized empirical risk minimization (ERM)}
\label{sec:erm}

In this section, we show how to apply the APCG method to solve
the regularized ERM problems associated with linear predictors.

Let $A_1, \ldots, A_n$ be vectors in $\reals^d$, $\phi_1$, \ldots, 
$\phi_n$ be a sequence of convex functions defined on~$\reals$,
and~$g$ be a convex function defined on $\reals^d$.
The goal of regularized ERM with linear predictors 
is to solve the following (convex) optimization problem:
\begin{equation}\label{eqn:erm-primal}
\minimize_{w\in\reals^d} 
~\left\{ P(w) ~\eqdef~ \frac{1}{n}\sum_{i=1}^n \phi_i(A_i^T w) + \lambda g(w) 
\right\}, 
\end{equation}
where $\lambda>0$ is a regularization parameter.
For binary classification, 
given a label $b_i\in\{\pm 1\}$ for each vector $A_i$, 
for $i=1,\ldots,n$, we obtain the linear SVM (support vector machine)
problem by setting
$\phi_i(z)=\max\{0, 1-b_i z\}$ and $g(w) = (1/2)\|w\|_2^2$.
Regularized logistic regression is obtained by setting 
$\phi_i(z)=\log(1+\exp(-b_i z))$. 
This formulation also includes regression problems.
For example, ridge regression is obtained by setting 
$\phi_i(z)=(1/2)(z-b_i)^2$ and $g(w) = (1/2)\|w\|_2^2$,
and we get the Lasso if $g(w)=\|w\|_1$.
Our method can also be extended to cases where each $A_i$ is a matrix, 
thus covering multiclass classification problems as well
(see, e.g., \cite{SSZhang13acclSDCA}).

For each $i=1,\ldots,n$, let $\phi_i^*$ be the convex conjugate of $\phi_i$, 
that is,
\[
    \phi_i^*(u)=\max_{z\in\reals} ~\{ z u - \phi_i(z)\}.
\]
The dual of the regularized ERM problem~\eqref{eqn:erm-primal},
which we call the primal, is to solve the problem 
(see, e.g., \cite{SSZhang13SDCA})
\begin{equation}\label{eqn:erm-dual}
    \maximize_{x\in\reals^n} 
    ~\left\{ D(x) ~\eqdef~ \frac{1}{n}\sum_{i=1}^n -\phi_i^*(-x_i)
    -\lambda g^*\left(\frac{1}{\lambda n} A x\right) \right\},
\end{equation}
where $A=[A_1,\ldots,A_n]$.
This is equivalent to minimize $F(x)\eqdef -D(x)$, that is,
\begin{equation}\label{eqn:erm-composite}
    \minimize_{x\in\reals^n} ~
    \left\{ F(x) ~\eqdef~
    \frac{1}{n}\sum_{i=1}^n \phi_i^*(-x_i)
    +\lambda g^*\left(\frac{1}{\lambda n} A x\right) 
\right\}.
\end{equation}
The structure of $F(x)$ above matches our general formulation of minimizing 
composite convex functions in~\eqref{eqn:composite-min} and~\eqref{eqn:Psi} 
with
\begin{equation}\label{eqn:simple-split}
f(x)=\lambda g^*\left( \frac{1}{\lambda n} A x \right),
    \qquad
    \Psi(x)=\frac{1}{n}\sum_{i=1}^n \phi_i^*(-x_i).
\end{equation}
Therefore, we can directly apply the APCG method to solve the 
problem~\eqref{eqn:erm-composite}, i.e., 
to solve the dual of the regularized ERM problem.
Here we assume that the proximal mappings of the conjugate functions $\phi_i^*$ 
can be computed efficiently, which is indeed the case for many regularized
ERM problems (see, e.g., \cite{SSZhang13SDCA,SSZhang13acclSDCA}).

In order to obtain accelerated linear convergence rates, we make the following
assumption.

\begin{assumption}\label{asmp:erm}
Each function $\phi_i$ is $1/\gamma$ smooth, and the function $g$ 
has unit convexity parameter~1.
\end{assumption}

Here we slightly abuse the notation by overloading~$\gamma$ and~$\lambda$, 
which appeared in Sections~\ref{sec:apcg} and~\ref{sec:analysis}.
In this section $\gamma$ represents the (inverse) smoothness parameter
of $\phi_i$, and~$\lambda$ denotes the regularization parameter on~$g$.
Assumption~\ref{asmp:erm} implies that each $\phi_i^*$ has strong
convexity parameter $\gamma$ (with respect to the local Euclidean norm)
and $g^*$ is differentiable and $\nabla g^*$ has Lipschitz constant~1.

In order to match the condition in Assumption~\ref{asmp:strong-convex},
i.e., $f(x)$ needs to be strongly convex, we can apply the technique 
in Section~\ref{sec:strong-psi} to relocate the strong convexity
from~$\Psi$ to~$f$.
Without loss of generality, we can use the following splitting
of the composite function $F(x)=f(x)+\Psi(x)$:
\begin{equation}\label{eqn:splitting}
f(x) = \lambda g^*\left(\frac{1}{\lambda n}A x\right)+\frac{\gamma}{2n}\|x\|_2^2,
\qquad
\Psi(x) = \frac{1}{n}\sum_{i=1}^n \left(\phi^*(-x_i)-\frac{\gamma}{2}\|x_i\|_2^2
\right).
\end{equation}
Under Assumption~\ref{asmp:erm}, the function $f$ is smooth and 
strongly convex and each $\Psi_i$, for $i=1,\ldots,n$, is still convex.
As a result, we have the following complexity guarantee 
when applying the APCG method to minimize the function $F(x)=-D(x)$.

\begin{theorem}\label{thm:dual-erm}
Suppose Assumption~\ref{asmp:erm} holds and $\|A_i\|_2\leq R$ 
for all $i=1,\ldots,n$.
In order to obtain an expected dual optimality gap
$\E[D^\star-D(x^{(k)})] \leq \epsilon$ using the APCG method,
it suffices to have
\begin{equation}\label{eqn:erm-complexity}
k \geq \left(n+\sqrt{\frac{n R^2}{\lambda\gamma}}\right)
\log(C/\epsilon).
\end{equation}
where $D^\star=\max_{x\in\reals^n} D(x)$ and 
\begin{equation}\label{eqn:C-def}
    C = D^\star-D(x^{(0)})+\frac{\gamma}{2n} \|x^{(0)}-x^\star\|_2^2.
\end{equation}
\end{theorem}

\begin{proof}
First, we notice that the function $f(x)$ defined in~\eqref{eqn:splitting}
is differentiable. Moreover, for any $x\in\reals^n$ and $h_i\in\reals$,
\begin{eqnarray*}
   \|\nabla_i f(x+U_i h_i)-\nabla_i f(x)\|_2
   &=& \left\|\frac{1}{n} A_i^T \left[\nabla g^*\left(\frac{1}{\lambda n} A(x+U_i h_i)\right) - \nabla g^*\left(\frac{1}{\lambda n} A x\right) \right] + 
   \frac{\gamma}{n}h_i \right\|_2 \\
   &\leq& \frac{\|A_i\|_2}{n} \left\|\nabla g^*\left(\frac{1}{\lambda n} A(x+U_i h_i)\right) - \nabla g^*\left(\frac{1}{\lambda n} A x\right) \right\|_2 + 
   \frac{\gamma}{n} \|h_i\|_2 \\
   &\leq& \frac{\|A_i\|_2}{n} \left\|\frac{1}{\lambda n} A_i h_i \right\|_2 + 
   \frac{\gamma}{n} \|h_i\|_2 \\
   &\leq& \left( \frac{\|A_i\|_2^2}{\lambda n^2} + \frac{\gamma}{n}\right)\|h_i\|_2,
\end{eqnarray*}
where the second inequality used the assumption that~$g$ 
has convexity parameter~$1$ and thus $\nabla g^*$ has Lipschitz constant~$1$.
The coordinate-wise Lipschitz constants as defined in 
Assumption~\ref{asmp:coord-smooth} are
\[
L_i~=~\frac{\|A_i\|_2^2}{\lambda n^2} + \frac{\gamma}{n} 
~\leq~ \frac{R^2+\lambda\gamma n}{\lambda n^2} ,
\qquad i=1,\ldots,n.
\]
The function~$f$ has convexity parameter $\frac{\gamma}{n}$ 
with respect to the Euclidean norm $\|\cdot\|_2$.
Let~$\mu$ be its convexity parameter with respect to the norm
$\|\cdot\|_L$ defined in~\eqref{eqn:L-norm}.
Then
\[
    \mu ~\geq~ \frac{\gamma}{n} \Big/ \frac{R^2+\lambda\gamma n}{\lambda n^2}
    ~=~ \frac{\lambda \gamma n}{R^2+\lambda \gamma n} .
\]
According to Theorem~\ref{thm:apcg-rate}, the APCG method converges 
geometrically:
\[
\E\left[D^\star-D(x^{(k)})\right]
~\leq~ \left(1-\frac{\sqrt{\mu}}{n}\right)^k  C
~\leq~ \exp\left(-\frac{\sqrt{\mu}}{n} k \right) C ,
\]
where the constant~$C$ is given in~\eqref{eqn:C-def}.
Therefore, in order to obtain $\E[D^\star-D(x^{(k)})] \leq \epsilon$,
it suffices to have the number of iterations~$k$ to be larger than
\[
    \frac{n}{\sqrt{\mu}} \log(C/\epsilon) 
    ~\leq~ n\sqrt{\frac{R^2+\lambda \gamma n}{\lambda \gamma n}} \log(C/\epsilon)
    ~=~ \sqrt{n^2 + \frac{n R^2}{\lambda \gamma}} \log(C/\epsilon) 
    ~\leq~ \left(n +\sqrt{\frac{n R^2}{\lambda\gamma}}\right) \log(C/\epsilon).
\]
This finishes the proof.
\end{proof}

Let us compare the result in Theorem~\ref{thm:dual-erm} with the complexity
of solving the dual problem~\eqref{eqn:erm-composite} using the accelerated 
full gradient (AFG) method of Nesterov \cite{Nesterov13composite}.
Using the splitting in~\eqref{eqn:simple-split} and under 
Assumption~\ref{asmp:erm}, the gradient $\nabla f(x)$ has Lipschitz
constant $\frac{\|A\|_2^2}{\lambda n^2}$, where $\|A\|_2$ denotes the spectral
norm of~$A$, and $\Psi(x)$
has convexity parameter $\frac{\gamma}{n}$ with respect to $\|\cdot\|_2$.
So the condition number of the problem is
\[
    \kappa =\frac{\|A\|_2^2}{\lambda n^2} \bigg/ \frac{\gamma}{n}
    =\frac{\|A\|_2^2}{\lambda\gamma n}.
\]
Suppose each iteration of the AFG method costs as much as~$n$ times
of the APCG method (as we will see in Section~\ref{sec:erm-impl}),
then the complexity of the AFG method \cite[Theorem~6]{Nesterov13composite}
measured in terms of number of coordinate gradient steps is
\[
  O\left(n \sqrt{\kappa}\log(1/\epsilon)\right)
  ~=~ O\left(\sqrt{\frac{n \|A\|_2^2}{\lambda\gamma}}\log(1/\epsilon)\right)
  ~\leq~ O\left( \sqrt{\frac{n^2 R^2}{\lambda\gamma}}\log(1/\epsilon) \right).
\]
The inequality above is due to $\|A\|_2^2\leq\|A\|_F^2\leq n R^2$.
Therefore in the ill-conditioned case 
(assuming $n\leq\frac{R^2}{\lambda\gamma}$),
the complexity of AFG can be a factor of~$\sqrt{n}$ worse than that of APCG.

Several state-of-the-art algorithms for regularized ERM, 
including SDCA \cite{SSZhang13SDCA},
SAG \cite{LeRouxSchmidtBach12,SchmidtLeRouxBach13} and
SVRG \cite{JohnsonZhang13,XiaoZhang14},
have the iteration complexity 
\[
O\left(\left(n +\frac{R^2}{\lambda\gamma}\right) \log(1/\epsilon)\right).
\]
Here the ratio $\frac{R^2}{\lambda\gamma}$ can be interpreted as the condition
number of the regularized ERM problem~\eqref{eqn:erm-primal} 
and its dual~\eqref{eqn:erm-dual}.
We note that our result in~\eqref{eqn:erm-complexity} can be much better 
for ill-conditioned problems, i.e., when the condition number 
$\frac{R^2}{\lambda \gamma}$ is much larger than~$n$.


Most recently, Shalev-Shwartz and Zhang \cite{SSZhang13acclSDCA}
developed an accelerated SDCA method which achieves the same complexity
$O\left(\left(n +\sqrt{\frac{n}{\lambda\gamma}}\right) \log(1/\epsilon)\right)$
as our method.
Their method is an inner-outer iteration procedure, where the outer loop
is a full-dimensional accelerated gradient method 
in the primal space $w\in\reals^d$.
At each iteration of the outer loop, the SDCA method \cite{SSZhang13SDCA}
is called to solve the dual problem~\eqref{eqn:erm-dual} with 
customized regularization parameter and precision.
In contrast, our APCG method is a straightforward single loop 
coordinate gradient method.

We note that the complexity bound for the aforementioned work are either
for the primal optimality $P(w^{(k)})-P^\star$ (SAG and SVRG) 
or for the primal-dual gap $P(w^{(k)})-D(x^{(k)})$ (SDCA and accelerated SDCA).
Our results in Theorem~\ref{thm:dual-erm} are in terms of the 
dual optimality $D^\star-D(x^{(k)})$.
In Section~\ref{sec:erm-primal}, we show how to recover primal solutions with
the same order of convergence rate.
In Section~\ref{sec:erm-impl}, we show how to exploit problem structure
of regularized ERM to compute the partial gradient $\nabla_i f(x)$, 
which together with the efficient implementation proposed 
in Section~\ref{sec:effi-impl}, completely avoid full-dimensional 
vector operations.
The experiments in Section~\ref{sec:experiments} illustrate that our method
has superior performance in reducing both the primal objective value
and the primal-dual gap.

\subsection{Recovering the primal solution}
\label{sec:erm-primal}

Under Assumption~\ref{asmp:erm}, the primal problem~\eqref{eqn:erm-primal}
and dual problem~\eqref{eqn:erm-dual} each has a unique solution, 
say $w^\star$ and $x^\star$, respectively.
Moreover, we have $P(w^\star)=D(x^\star)$.
With the definition
\begin{equation}\label{eqn:omega}
    \omega(x) = \nabla g^*\left(\frac{1}{\lambda n} A x\right),
\end{equation}
we have $w^\star=\omega(x^\star)$.
When applying the APCG method to solve the dual regularized ERM problem, 
which generate a dual sequence $x^{(k)}$, 
we can obtain a primal sequence $w^{(k)} = \omega(x^{(k)})$.
Here we discuss the relationship between the primal-dual gap 
$P(w^{(k)})-D(x^{(k)})$ and the dual optimality $D^\star-D(x^{(k)})$.

Let $a=(a_1,\ldots,a_n)$ be a vector in $\reals^n$.
We consider the saddle-point problem
\begin{equation}\label{spp}
\max_{x}~\min_{a,w}~\left\{ \Phi(x,a,w)~\eqdef~
\frac{1}{n}\sum_{i=1}^n\phi_i(a_i)+\lambda g(w)-\frac{1}{n}\sum_{i=1}^n
x_i(A_i^T w-a_i) \right\},
\end{equation}
so that
\begin{eqnarray*}
D(x)~=~\min_{a, w} ~\Phi(x,a, w).
\end{eqnarray*}
Given an approximate dual solution $x^{(k)}$ (generated by the APCG method),
we can find a pair of primal solutions 
$(a^{(k)},w^{(k)})=\argmin_{a,w}\Phi(x^{(k)},a,w)$, 
or more specifically,
\begin{eqnarray}
\label{pfromd1}
a^{(k)}_i &=& \arg\max_{a_i} \left\{ -x^{(k)}_i a_i-\phi_i(a_i) \right\}
~\in~ \partial\phi_i^*(-x_i^{(k)}),
\quad i=1,\ldots,n, \\
\label{pfromd2}
w^{(k)} &=& \arg\max_{w} \left\{ w^T\left(\frac{1}{\lambda n}Ax^{(k)}\right)
-g(w)\right\}
~=~ \nabla g^*\left(\frac{1}{\lambda n} A x^{(k)} \right) .
\end{eqnarray}
As a result, we obtain a subgradient of~$D$ at $x^{(k)}$, 
denoted $D'(x^{(k)})$, and
\begin{eqnarray}
\label{dualgrad}
\|D'(x^{(k)})\|_2^2=\frac{1}{n^2}\sum_{i=1}^n
\left(A_i^T w^{(k)}-a^{(k)}_i\right)^2.
\end{eqnarray}
We note that $\|D'(x^{(k)})\|_2^2$ is not only a measure of the dual optimality 
of $x^{(k)}$, 
but also a measure of the primal feasibility of $(a^{(k)},w^{(k)})$.
In fact, it can also bound the primal-dual gap, 
which is the result of the following lemma.

\begin{lemma}
\label{pdgbound}
Given any dual solution $x^{(k)}$, let $(a^{(k)},w^{(k)})$ be defined 
as in~\eqref{pfromd1} and~\eqref{pfromd2}. Then
\[
P(w^{(k)})-D(x^{(k)}) ~\leq~
\frac{1}{2n\gamma}\sum_{i=1}^n \left(A_i^T w^{(k)}-a^{(k)}_i\right)^2
~=~\frac{n}{2\gamma}\|D'(x^{(k)})\|_2^2.
\]
\end{lemma}
\begin{proof}
Because of \eqref{pfromd1}, we have $\nabla\phi_i(a^{(k)}_i)=-x^{(k)}_i$.
The $1/\gamma$-smoothness of $\phi_i(a)$ implies
\begin{eqnarray*}
P(w^{(k)})&=&\frac{1}{n}\sum_{i=1}^n \phi_i(A_i^T w^{(k)})+\lambda g(w^{(k)})\\
&\leq&\frac{1}{n}\sum_{i=1}^n \left(\phi_i(a^{(k)}_i)+\nabla \phi_i(a^{(k)}_i)^T\left(A_i^T w^{(k)}-a^{(k)}_i\right)+\frac{1}{2\gamma}\left(A_i^T w^{(k)}-a^{(k)}_i\right)^2\right) + \lambda g(w^{(k)})\\
&=&\frac{1}{n}\sum_{i=1}^n \left(\phi_i(a^{(k)}_i)-x^{(k)}_i\left(A_i^T w^{(k)}-a^{(k)}_i\right)+\frac{1}{2\gamma}\left(A_i^T w^{(k)}-a^{(k)}_i\right)^2\right) + \lambda g(w^{(k)})\\
&=&\Phi(x^{(k)},a^{(k)},w^{(k)})+\frac{1}{2n\gamma}\sum_{i=1}^n \left(A_i^T w^{(k)}-a^{(k)}_i\right)^2\\
&=&D(x^{(k)})+\frac{1}{2n\gamma}\sum_{i=1}^n (A_i^T w^{(k)}-a^{(k)}_i)^2,
\end{eqnarray*}
which leads to the inequality in the conclusion. 
The equality in the conclusion is due to~\eqref{dualgrad}.
\end{proof}

The following theorem states that under a stronger assumption than
Assumption~\ref{asmp:erm}, the primal-dual gap can be bounded directly
by the dual optimality gap, hence
they share the same order of convergence rate.

\begin{theorem}\label{thm:gap-by-dual}
Suppose $g$ is $1$-strongly convex and each $\phi_i$ is $1/\gamma$-smooth and 
also $1/\eta$-strongly convex (all with respect to the Euclidean norm
$\|\cdot\|_2$).
Given any dual point $x^{(k)}$, let the primal correspondence be
$w^{(k)}=\omega(x^{(k)})$, i.e., generated from~\eqref{pfromd2}. 
Then we have
\begin{equation}\label{eqn:gap-by-dual}
P(w^{(k)})-D(x^{(k)})
~\leq~ \frac{\lambda\eta n+\|A\|_2^2}{\lambda\gamma n}
\left(D^\star-D(x^{(k)})\right),
\end{equation}
where $\|A\|_2$ denotes the spectral norm of~$A$.
\end{theorem}

\begin{proof}
Since $g(w)$ is $1$-strongly convex, the function
$f(x)=\lambda g^*\left(\frac{A x}{\lambda n} \right)$ 
is differentiable and $\nabla f(x)$
has Lipschitz constant $\frac{\|A\|_2^2}{\lambda n^2}$.
Similarly, since each $\phi_i$ is $1/\eta$ strongly convex, 
the function $\Psi(x)=\frac{1}{n}\sum_{i=1}^n \phi_i^*(-x_i)$
is differentiable and $\nabla \Psi(x)$ has Lipschitz constant $\frac{\eta}{n}$.
Therefore, the function $-D(x)=f(x)+\Psi(x)$ is smooth and
its gradient has Lipschitz constant
\[
    \frac{\|A\|_2^2}{\lambda n^2} + \frac{\eta}{n}
    =\frac{\lambda\eta n+\|A\|_2^2}{\lambda n^2}.
\]
It is known that (e.g., \cite[Theorem~2.1.5]{Nesterov04book}) 
if a function $F(x)$ is convex and $L$-smooth, then
$$
F(y)\geq F(x)+\nabla F(x)^T(y-x)+\frac{1}{2L}\|\nabla F(x)-\nabla F(y)\|_2^2
$$
for all $x,y\in\reals^n$. 
Applying the above inequality to $F(x)=-D(x)$, we get for all $x$ and $y$,
\begin{equation}\label{eqn:smooth-D}
-D(y)\geq -D(x)-\nabla D(x)^T(y-x)+\frac{\lambda n^2}{2(\lambda\eta n+\|A\|_2^2)}\|\nabla D(x)-\nabla D(y)\|_2^2 .
\end{equation}

Under our assumptions, the saddle-point problem~\eqref{spp} has a unique
solution $(x^\star,a^\star,w^\star)$,
where $w^\star$ and $x^\star$ are the solutions to the primal and dual 
problems~\eqref{eqn:erm-primal} and~\eqref{eqn:erm-dual}, respectively.
Moreover, they satisfy the optimality conditions
\[
    A_i^T w^\star-a^\star_i=0, \qquad
    a^\star_i=\nabla\phi_i^*(-x^\star_i), \qquad
    w^\star=\nabla g^* \left(\frac{1}{\lambda n} A x^\star \right).
\]
Since~$D$ is differentiable in this case, we have $D'(x)=\nabla D(x)$ and
$\nabla D(x^\star)=0$.
Now we choose~$x$ and~$y$ in~\eqref{eqn:smooth-D} to be $x^\star$ and $x^{(k)}$
respectively.
This leads to
\begin{eqnarray*}
\|\nabla D(x^{(k)})\|_2^2~=~\|\nabla D(x^{(k)})-\nabla D(x^\star)\|_2^2
~\leq~ \frac{2(\lambda \eta n+\|A\|_2^2)}{\lambda n^2}(D(x^\star)-D(x^{(k)})).
\end{eqnarray*}
Then the conclusion can be derived from Lemma \ref{pdgbound}.
\end{proof}

The assumption that each $\phi_i$ is $1/\gamma$-smooth and $1/\eta$-strongly
convex implies that $\gamma\leq\eta$. 
Therefore the coefficient on the right-hand side of~\eqref{eqn:gap-by-dual}
satisfies
$
\frac{\lambda\eta n+\|A\|_2^2}{\lambda\gamma n} > 1 .
$
This is consistent with the fact that for any pair of primal and dual points
$w^{(k)}$ and $x^{(k)}$, we always have
$P(w^{(k)})-D(x^{(k)}) \geq D^\star - D(x^{(k)})$.

\begin{corollary}\label{cor:priml-dual}
Under the assumptions of Theorem~\ref{thm:gap-by-dual},
in order to obtain an expected primal-dual gap 
$\E\left[P(w^{(k)})-D(x^{(k)})\right]\leq\epsilon$ using the APCG method,
it suffices to have 
\[
k \geq \left(n+\sqrt{\frac{n R^2}{\lambda\gamma}}\right)
\log\left(\frac{(\lambda\eta n+\|A\|_2^2)}{\lambda\gamma n} 
\frac{C}{\epsilon}\right) ,
\]
where 
the constant $C$ is defined in~\eqref{eqn:C-def}.
\end{corollary}

The above results require that each~$\phi_i$ be both smooth and
strongly convex. 
One example that satisfies such assumptions is ridge regression,
where $\phi_i(a_i)=\frac{1}{2}(a_i-b_i)^2$ and $g(w)=\frac{1}{2}\|w\|_2^2$.
For problems that only satisfy Assumption~\ref{asmp:erm}, we may add a small
strongly convex term $\frac{1}{2\eta}(A_i^T w)^2$ to each loss
$\phi_i(A_i^T w)$, and obtain that the primal-dual gap 
(of a slightly perturbed problem)
share the same accelerated linear convergence rate as the dual optimality gap. 
Alternatively, we can obtain the same guarantee with the extra cost of
a proximal full gradient step. This is summarized in the following theorem.

\begin{theorem}\label{thm:gap-by-dual-2}
Suppose Assumption~\ref{asmp:erm} holds. 
Given any dual point $x^{(k)}$, define
\begin{equation}\label{eqn:full-prox}
 T(x^{(k)}) = \arg\min_{x\in\reals^n} \left\{ \langle \nabla f(x^{(k)}),
 x \rangle + \frac{\|A\|_2^2}{2\lambda n^2}\|x-x^{(k)}\|_2^2 + \Psi(x) \right\},
\end{equation}
where~$f$ and~$\Psi$ are defined in the simple 
splitting~\eqref{eqn:simple-split}.
Let
\begin{equation}\label{eqn:full-prox-w}
    w^{(k)} = \omega(T(x^{(k)})) 
    = \nabla g^*\left(\frac{1}{\lambda n} A T(x^{(k)})\right).
\end{equation}
Then we have
\begin{equation}\label{eqn:gap-by-dual-2}
P(w^{(k)})-D(T(x^{(k)}))\leq \frac{4\|A\|_2^2}{\lambda\gamma n}
\left(D(x^\star)-D(x^{(k)})\right).
\end{equation}
\end{theorem}

\begin{proof}
Notice that the Lipschitz constant of $\nabla f(x)$ is 
$L_f=\frac{\|A\|_2^2}{\lambda n^2}$, which is used in calculating $T(x^{(k)})$.
The corresponding gradient mapping \cite{Nesterov13composite} at $x^{(k)}$ is 
$$
G(x^{(k)})= L_f \left(x^{(k)}-T(x^{(k)})\right)
= \frac{\|A\|_2^2}{\lambda n^2}\left(x^{(k)} - T(x^{(k)})\right).
$$ 
According to \cite[Theorem~1]{Nesterov13composite}, we have 
$$
\left\|D'\left(T(x^{(k)})\right)\right\|_2^2 
\leq 4\left\|G(x^{(k)})\right\|_2^2 
\leq 8 L_f\left(D(x^\star)-D(x^{(k)})\right)
= \frac{8\|A\|_2^2}{\lambda n^2}\left(D(x^\star)-D(x^{(k)})\right).
$$
The conclusion can then be derived from Lemma \ref{pdgbound}.
\end{proof}

Here the coefficient in the right-hand side of~\eqref{eqn:gap-by-dual-2},
$\frac{4\|A\|_2^2}{\lambda \gamma n}$, can be less than~$1$.
This does not contradict with the fact that the primal-dual gap should be no less
than the dual optimality gap, because the primal-dual gap on the left-hand side
of~\eqref{eqn:gap-by-dual-2} is measured at $T(x^{(k)})$ rather than
$x^{(k)}$.

\begin{corollary}\label{cor:priml-dual-extra}
Suppose Assumption~\ref{asmp:erm} holds.
In order to obtain a primal-dual pair $w^{(k)}$ and $x^{(k)}$ such that
$\E\left[P(w^{(k)})-D(T(x^{(k)}))\right]\leq\epsilon$,
it suffices to run the APCG method for 
\[
k \geq \left(n+\sqrt{\frac{n R^2}{\lambda\gamma}}\right)
\log\left(\frac{4\|A\|_2^2}{\lambda\gamma n}\frac{C}{\epsilon}\right)
\]
steps and follow with a proximal full gradient step~\eqref{eqn:full-prox}
and~\eqref{eqn:full-prox-w},
where $C$ is defined in~\eqref{eqn:C-def}.
\end{corollary}

We note that the computational cost of the proximal full gradient 
step~\eqref{eqn:full-prox} is comparable with~$n$ proximal coordinate gradient
steps.
Therefore the overall complexity of of this scheme is on the same order
as necessary for the expected dual optimality gap to reach~$\epsilon$.
Actually the numerical experiments in Section~\ref{sec:experiments}
show that running the APCG method alone without the final full gradient step
is sufficient to reduce the primal-dual gap at a very fast rate.

\subsection{Implementation details}
\label{sec:erm-impl}

Here we show how to exploit the structure of the regularized ERM problem
to efficiently compute the coordinate gradient $\nabla_{i_k} f(y^{(k)})$,
and totally avoid full-dimensional updates in Algorithm~\ref{alg:apcg-mu-effi}.

\begin{algorithm}[t]
\caption{APCG for solving regularized ERM with $\mu>0$}
\label{alg:apcg-erm}
\textbf{input:} $x^{(0)}\in\dom(\Psi)$ and convexity parameter 
$\mu=\frac{\lambda\gamma n}{R^2+\lambda\gamma n}$.\\[0.5ex]
\textbf{initialize:} 
set $\alpha=\frac{\sqrt{\mu}}{n}$ and $\rho=\frac{1-\alpha}{1+\alpha}$,
and let
$u^{(0)}=0$, $v^{(0)}=x^{(0)}$, $p^{(0)}=0$ and $q^{(0)}=Ax^{(0)}$.
\\[0.5ex]
\textbf{iterate:} repeat for $k=0,1,2,\ldots$ 
\begin{enumerate} \itemsep 0pt 
\item Choose $i_k\in\{1,\ldots,n\}$ uniformly at random, compute
    the coordinate gradient 
    \[
    \nabla^{(k)}_{i_k} = \frac{1}{\lambda n^2} 
    \left(\rho^{k+1} A_{i_k}^T p^{(k)} + A_{i_k}^T q^{(k)}\right)
    + \frac{\gamma}{n} \left(\rho^{k+1} u^{(k)}_{i_k} + v^{(k)}_{i_k} \right) .
    \]
\item Compute coordinate increment 
\begin{equation}\label{eqn:erm-h}
    h^{(k)}_{i_k} = \argmin_{h\in\reals^{N_{i_k}}} 
    \left\{ \frac{\alpha(\|A_{i_k}\|^2+\lambda\gamma n)}{2\lambda n} \|h \|_2^2 
    + \langle \nabla^{(k)}_{i_k} , h\rangle 
    + \Psi_{i_k}\left(-\rho^{k+1}u^{(k)}_{i_k}+v^{(k)}_{i_k}+h\right)
\right\}.
\end{equation}
\item 
Let $u^{(k+1)} = u^{(k)}$ and $v^{(k+1)} = v^{(k)}$, and update
\begin{eqnarray}
u^{(k+1)}_{i_k} = u^{(k)}_{i_k} - \frac{1-n\alpha}{2\rho^{k+1}} h^{(k)}_{i_k},
~~~~&\qquad&
v^{(k+1)}_{i_k} = v^{(k)}_{i_k} + \frac{1+n\alpha}{2} h^{(k)}_{i_k},
\nonumber\\[1ex]
p^{(k+1)} = p^{(k)} - \frac{1-n\alpha}{2\rho^{k+1}} A_{i_k} h^{(k)}_{i_k},
&\qquad&
q^{(k+1)} = q^{(k)} + \frac{1+n\alpha}{2} A_{i_k} h^{(k)}_{i_k} . 
\label{eqn:p-q-update}
\end{eqnarray}
\end{enumerate}
\textbf{output:} approximate dual and primal solutions
\[
x^{(k+1)}=\rho^{k+1} u^{(k+1)} + v^{(k+1)} , \qquad
w^{(k+1)} = \frac{1}{\lambda n} \left(\rho^{k+1} p^{(k+1)} + q^{(k+1)}\right).
\]
\end{algorithm}

We focus on the special case $g(w)=\frac{1}{2}\|w\|_2^2$
and show how to compute $\nabla_{i_k} f(y^{(k)})$.
In this case, $g^*(v) = \frac{1}{2}\|v\|_2^2$ and $\nabla g^*(\cdot)$ is
the identity map.
According to~\eqref{eqn:splitting}, 
\[
    \nabla_{i_k} f(y^{(k)}) =\frac{1}{\lambda n^2} A_{i_k}^T (A y^{(k)}) 
    +\frac{\gamma}{n} y^{(k)}_{i_k} .
\]
Notice that we do not form $y^{(k)}$ in Algorithm~\ref{alg:apcg-mu-effi}.
By Proposition~\ref{prop:equiv}, we have 
\[
    y^{(k)}=\rho^{k+1} u^{(k)}+v^{(k)}.
\]
So we can store and update the two vectors
\[
   p^{(k)}=A u^{(k)}, \qquad q^{(k)}=A v^{(k)},
\]
and obtain
\[
    A y^{(k)} = \rho^{k+1} p^{(k)} + q^{(k)}.
\]
Since the update of both $u^{(k)}$ and $v^{(k)}$ at each iteration only
involves the single coordinate~$i_k$, we can update $p^{(k)}$ and $q^{(k)}$ 
by adding or subtracting a scaled column $A_{i_k}$,
as given in~\eqref{eqn:p-q-update}.
The resulting method is detailed in Algorithm~\ref{alg:apcg-erm}.

In Algorithm~\ref{alg:apcg-erm}, we use $\nabla^{(k)}_{i_k}$ to represent
$\nabla_{i_k} f(y^{(k)})$ to reflect the fact that we never form~$y^{(k)}$
explicitly.
The function $\Psi_i$ in~\eqref{eqn:erm-h} is the one given 
in~\eqref{eqn:splitting}, i.e., 
\[
    \Psi_i(x_i) = \frac{1}{n} \phi_i^*(-x_i) - \frac{\gamma}{2n}\|x_i\|_2^2.
\]
Each iteration of Algorithm~\ref{alg:apcg-erm} only involves 
the two inner products $A_{i_k}^T p^{(k)}$ and $A_{i_k}^T q^{(k)}$ 
in computing $\nabla^{(k)}_{i_k}$,
and the two vector additions in~\eqref{eqn:p-q-update}.
They all cost $O(d)$ rather than $O(n)$.
When the $A_i$'s are sparse (the case of most large-scale
problems), these operations can be carried out very efficiently.
Basically, each iteration of Algorithm~\ref{alg:apcg-erm} only 
cost twice as much as that of SDCA \cite{HCLKS08,SSZhang13SDCA}.

In Step~3 of Algorithm~\ref{alg:apcg-erm}, the division by $\rho^{k+1}$ in
updating $u^{(k)}$ and $p^{(k)}$ may cause numerical problems 
because $\rho^{k+1}\to 0$ as the number of iterations~$k$ getting large.
To fix this issue, we notice that $u^{(k)}$ and $p^{(k)}$ are always accessed
in Algorithm~\ref{alg:apcg-erm} in the forms of
$\rho^{k+1}u^{(k)}$ and $\rho^{k+1}p^{(k)}$.
So we can replace $u^{(k)}$ and $p^{(k)}$ by
\[
    \bar u^{(k)}=\rho^{k+1}u^{(k)}, \qquad
    \bar p^{(k)}=\rho^{k+1}p^{(k)},
\]
which can be updated without numerical problem.
To see this, we have
\begin{eqnarray*}
    \bar u^{(k+1)} &=&  \rho^{k+2} u^{(k+1)} \\
    &=& \rho^{k+2} \left( u^{(k)} - \frac{1-n\alpha}{2\rho^{k+1}} U_{i_k}h^{(k)}_{i_k}\right) \\
    &=& \rho\left( \bar u^{(k)} - \frac{1-n\alpha}{2} U_{i_k}h^{(k)}_{i_k}\right).
\end{eqnarray*}
Similarly, we have
\[
    \bar p^{(k+1)} = \rho \left(\bar p^{(k)} - \frac{1-n\alpha}{2} A_{i_k} h^{(k)}_{i_k} \right).
\]

\subsection{Numerical experiments}
\label{sec:experiments}

In our experiments, we solve the regularized ERM problem~\eqref{eqn:erm-primal}
with smoothed hinge loss for binary classification.
That is, we pre-multiply each feature vector $A_i$ by its label 
$b_i\in\{\pm 1\}$ and let
        \[
            \phi_i(a) = \left\{ \begin{array}{ll} 0 & \mathrm{if}~ a \geq 1,\\
                1-a-\frac{\gamma}{2} & \mathrm{if}~ a\leq 1-\gamma,\\
                \frac{1}{2\gamma}(1-a)^2 & \mathrm{otherwise,}
            \end{array} \right.
            \qquad i=1,\ldots,n.
        \]
The conjugate function of $\phi_i$ is 
$\phi_i^*(b)=b+\frac{\gamma}{2}b^2$ if $b\in[-1,0]$ and $\infty$ otherwise.
Therefore we have 
\[
\Psi_i(x_i) = \frac{1}{n}\left(\phi_i^*(-x_i)-\frac{\gamma}{2}\|x_i\|_2^2\right) 
            = \left\{ \begin{array}{ll} \frac{-x_i}{n} & \mathrm{if}~ x_i\in[0,1]\\ \infty & \mathrm{otherwise.} \end{array} \right.
\]
For the regularization term, we use $g(w)=\frac{1}{2}\|w\|_2^2$.
We used three publicly available datasets obtained from \cite{LIBSVMdata}.
The characteristics of these datasets are summarized 
in Table~\ref{tab:datasets}.

\begin{table}[b]
\begin{center}
    \begin{tabular}{|r|r|r|r|r|}
        \hline
datasets & source & number of samples $n$ & number of features $d$ & sparsity \\
        \hline
        {RCV1}   & \cite{RCV1} & 20,242 &  47,236 & 0.16\% \\
        {covtype}& \cite{covertype} & 581,012 &  54 & 22\%  \\ 
        {News20} & \cite{News20binary,Lang95News20} & 19,996 &  1,355,191 & 0.04\%  \\
        \hline
    \end{tabular}
    \caption{Characteristics of three binary classification datasets
        obtained from \cite{LIBSVMdata}.}
    \label{tab:datasets}
\end{center}
\end{table}

In our experiments, we comparing the APCG method (Algorithm~\ref{alg:apcg-erm})
with SDCA \cite{SSZhang13SDCA} 
and the accelerated full gradient method (AFG) \cite{Nesterov04book} 
with and additional line search procedure to improve efficiency.
When the regularization parameter $\lambda$ is not too small
(around $10^{-4}$), then APCG performs similarly as SDCA as predicted
by our complexity results,
and they both outperform AFG by a substantial margin.

Figure~\ref{fig:primal} shows the reduction of primal optimality
$P(w^{(k)})-P^\star$ by
the three methods in the ill-conditioned setting, 
with $\lambda$ varying form $10^{-5}$ to $10^{-8}$.
For APCG, the primal points $w^{(k)}$ are generated simply as 
$w^{(k)}=\omega(x^{(k)})$ defined in~\eqref{eqn:omega}.
Here we see that APCG has superior
performance in reducing the primal objective value compared with SDCA and
AFG, even 
without performing the final proximal full gradient step described in 
Theorem~\ref{thm:gap-by-dual-2}.

Figure~\ref{fig:gap} shows the reduction of primal-dual gap
$P(w^{(k)})-D(x^{(k)})$ by the two methods APCG and SDCA.
We can see that in the ill-conditioned setting,
the APCG method is more effective in reducing the primal-dual gap as well. 

\begin{figure}[p]
    \begin{tabular}[h]{@{}c|ccc@{}}
    $\lambda$ & RCV1 & covertype & News20 \\
    \hline \\
    \raisebox{10ex}{$10^{-5}$} 
        &\quad \includegraphics[width=0.28\textwidth]{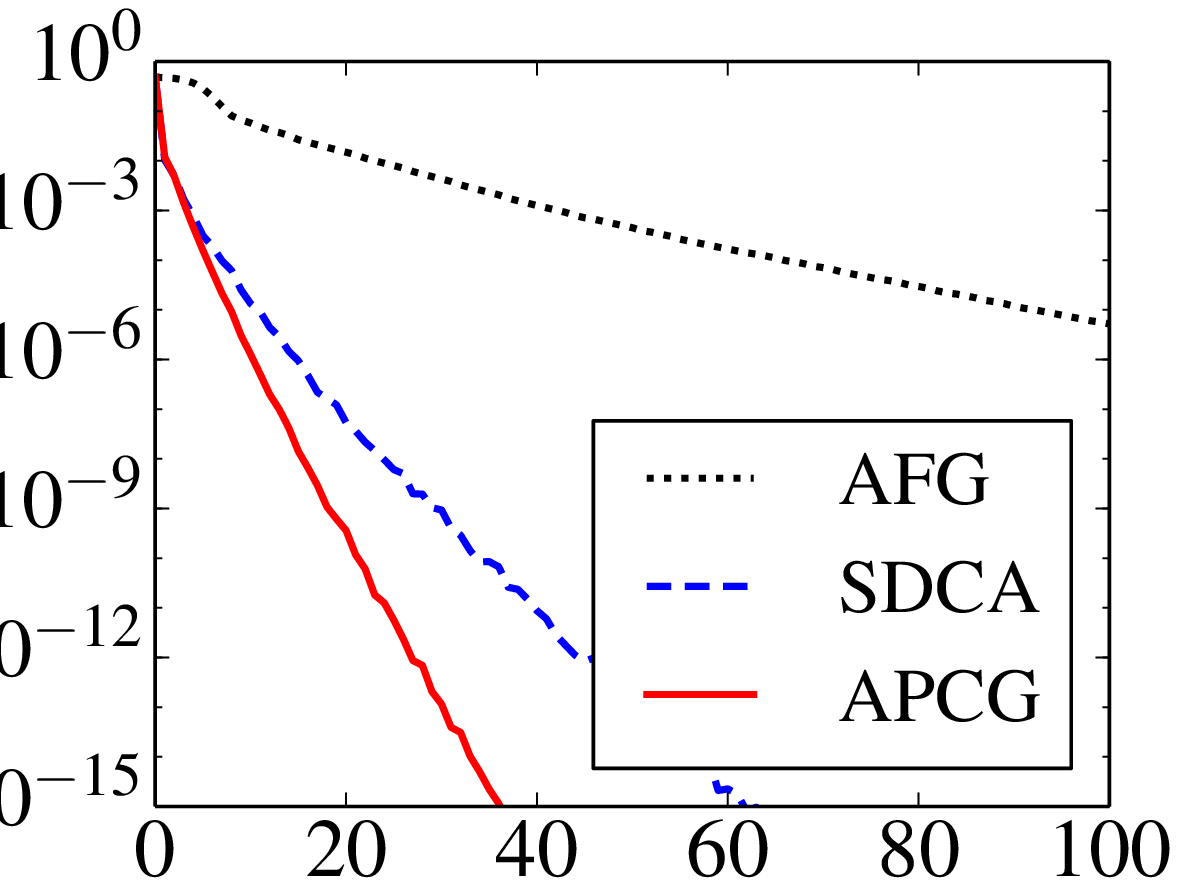} 
        & \includegraphics[width=0.28\textwidth]{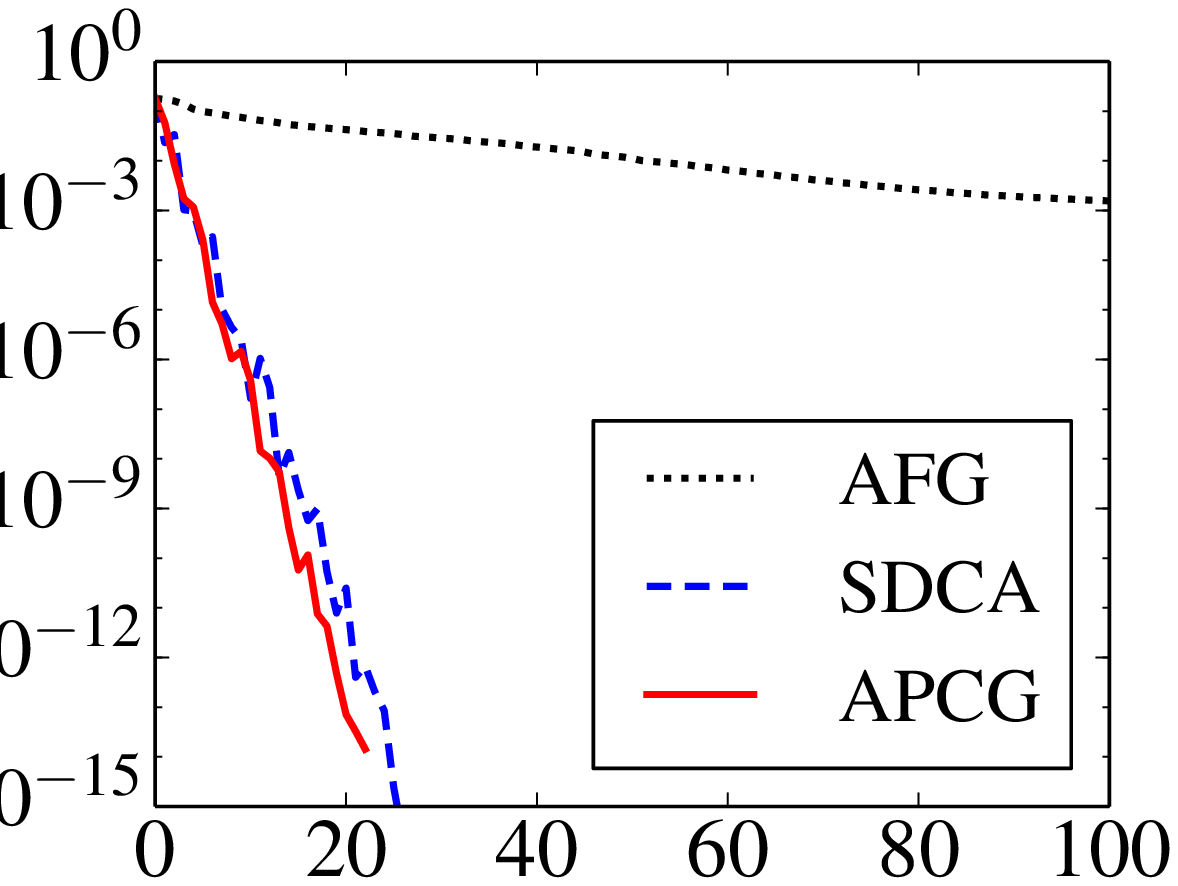} 
        & \includegraphics[width=0.28\textwidth]{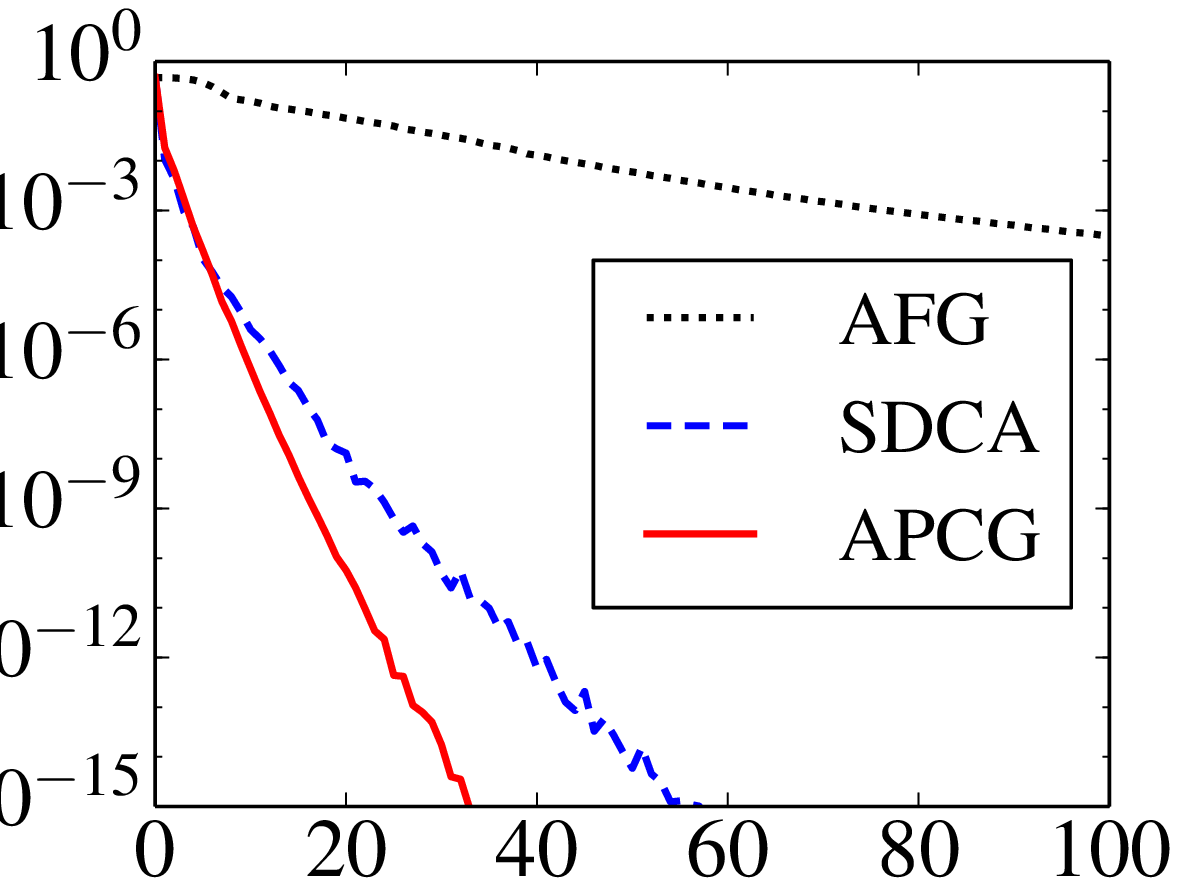}\\
    \raisebox{10ex}{$10^{-6}$} 
        &\quad \includegraphics[width=0.28\textwidth]{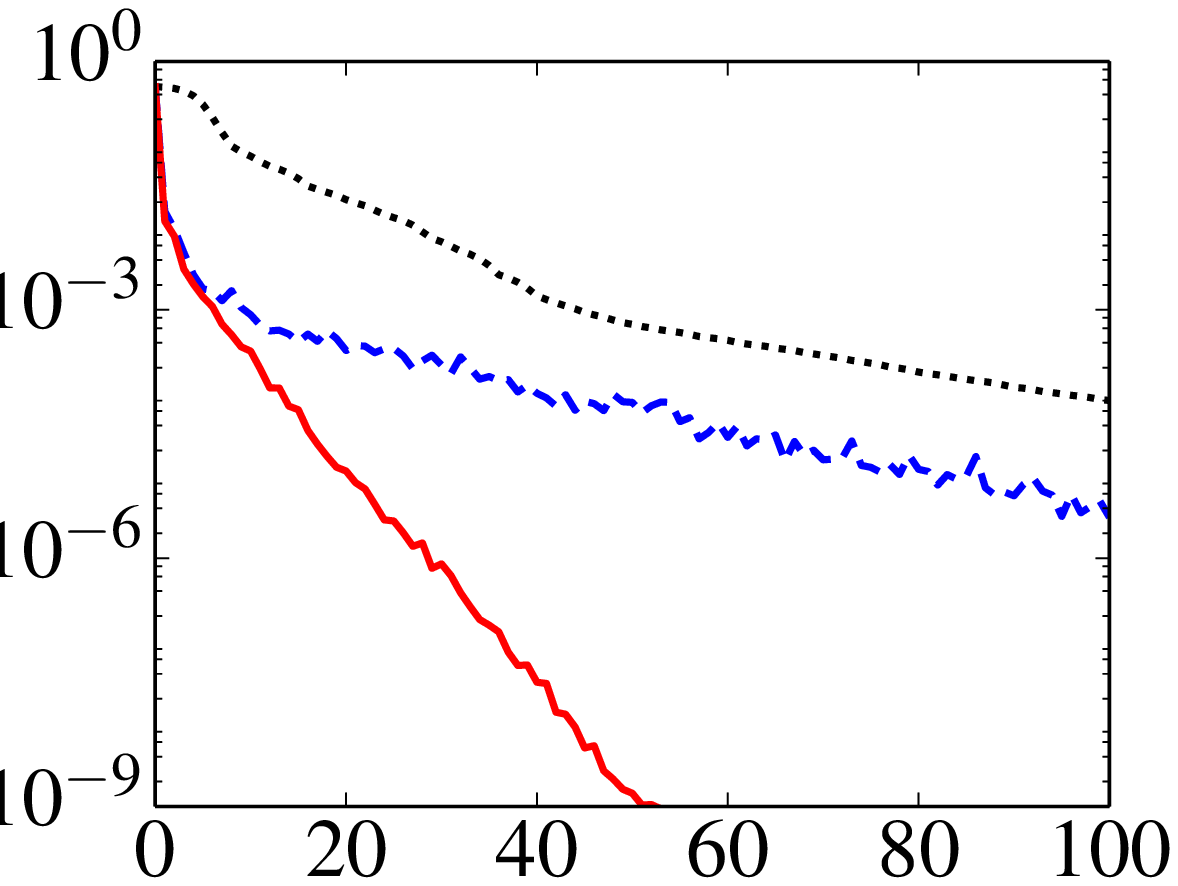} 
        & \includegraphics[width=0.28\textwidth]{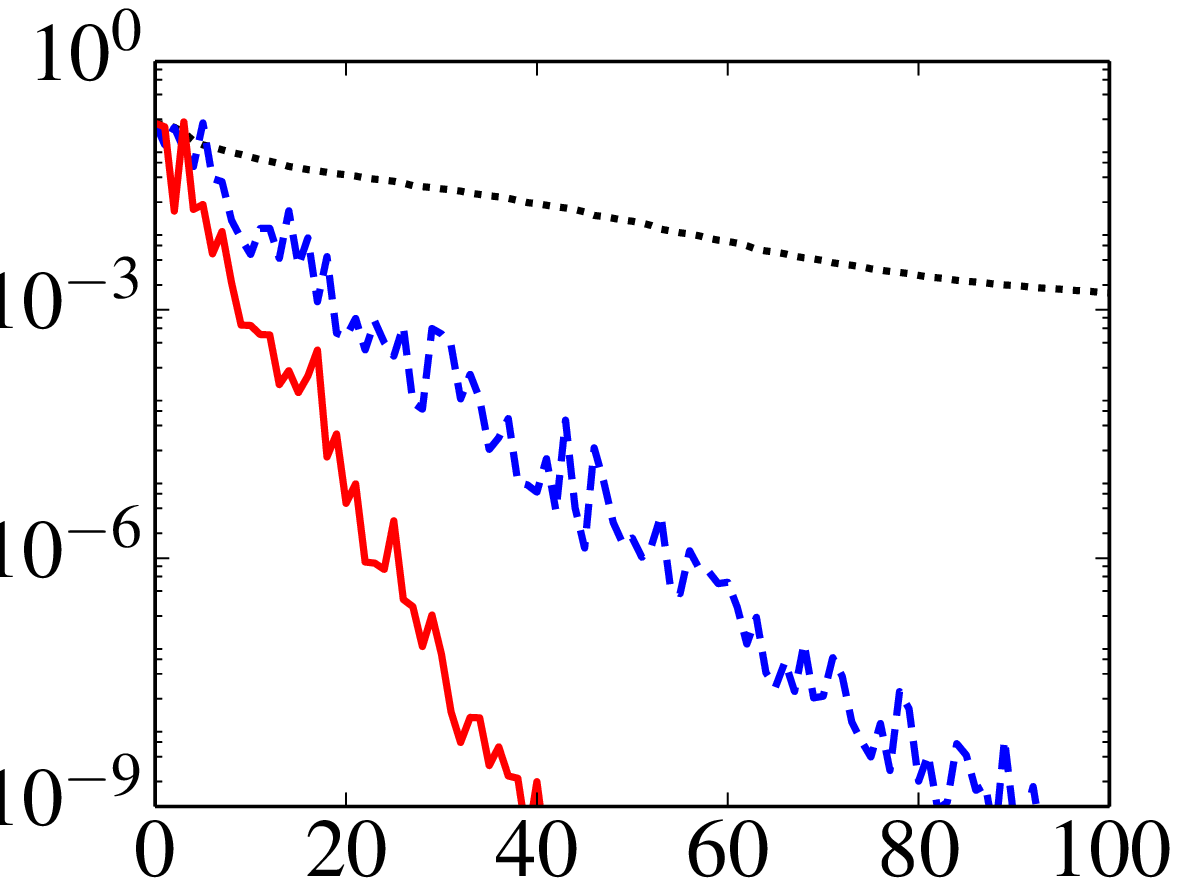} 
        & \includegraphics[width=0.28\textwidth]{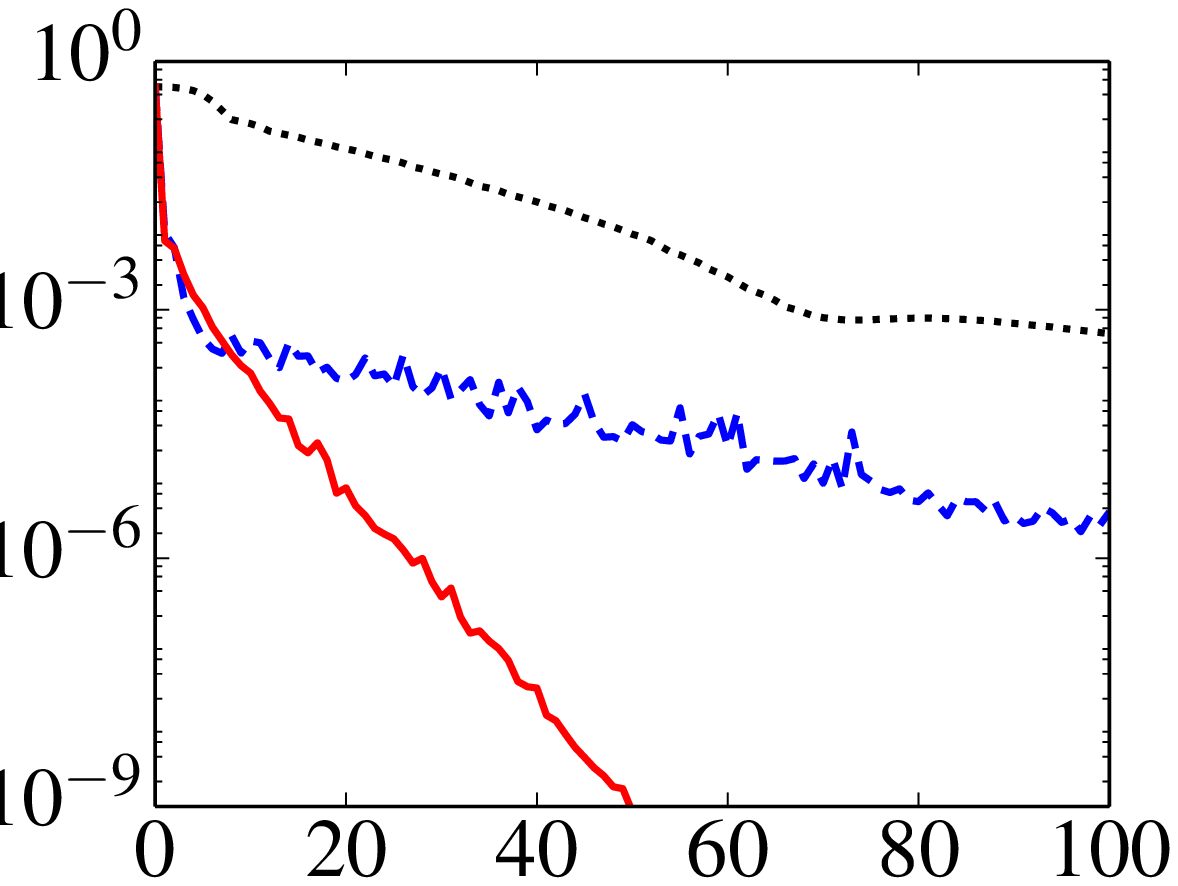}\\
    \raisebox{10ex}{$10^{-7}$} 
        &\quad \includegraphics[width=0.28\textwidth]{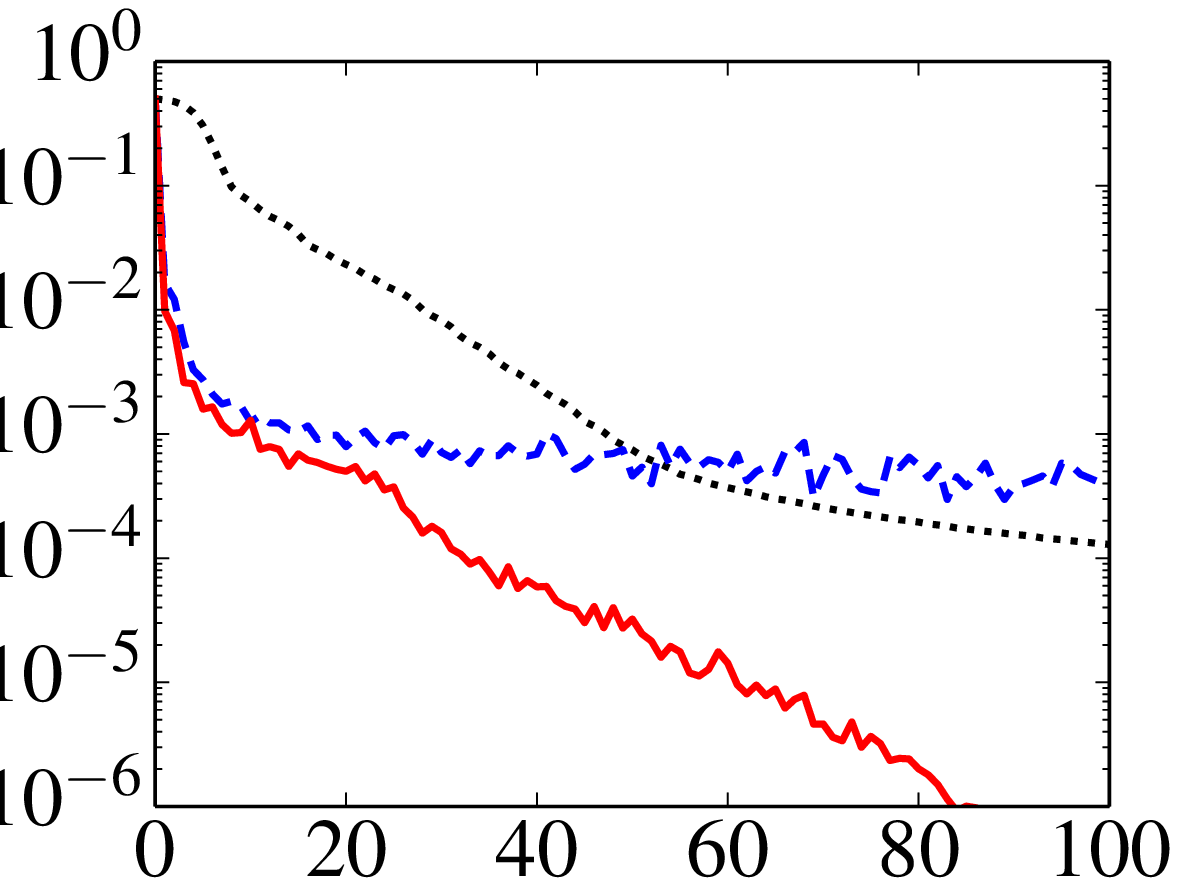} 
        & \includegraphics[width=0.28\textwidth]{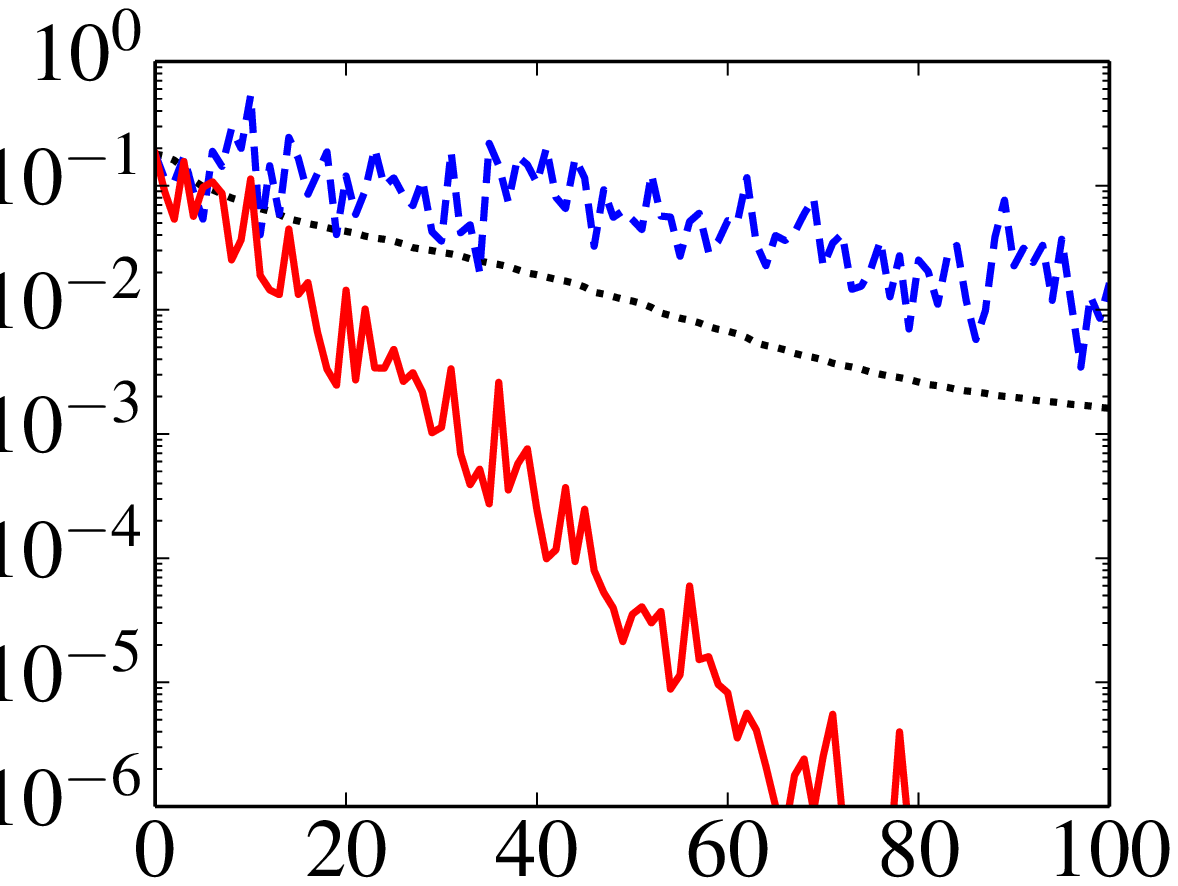} 
        & \includegraphics[width=0.28\textwidth]{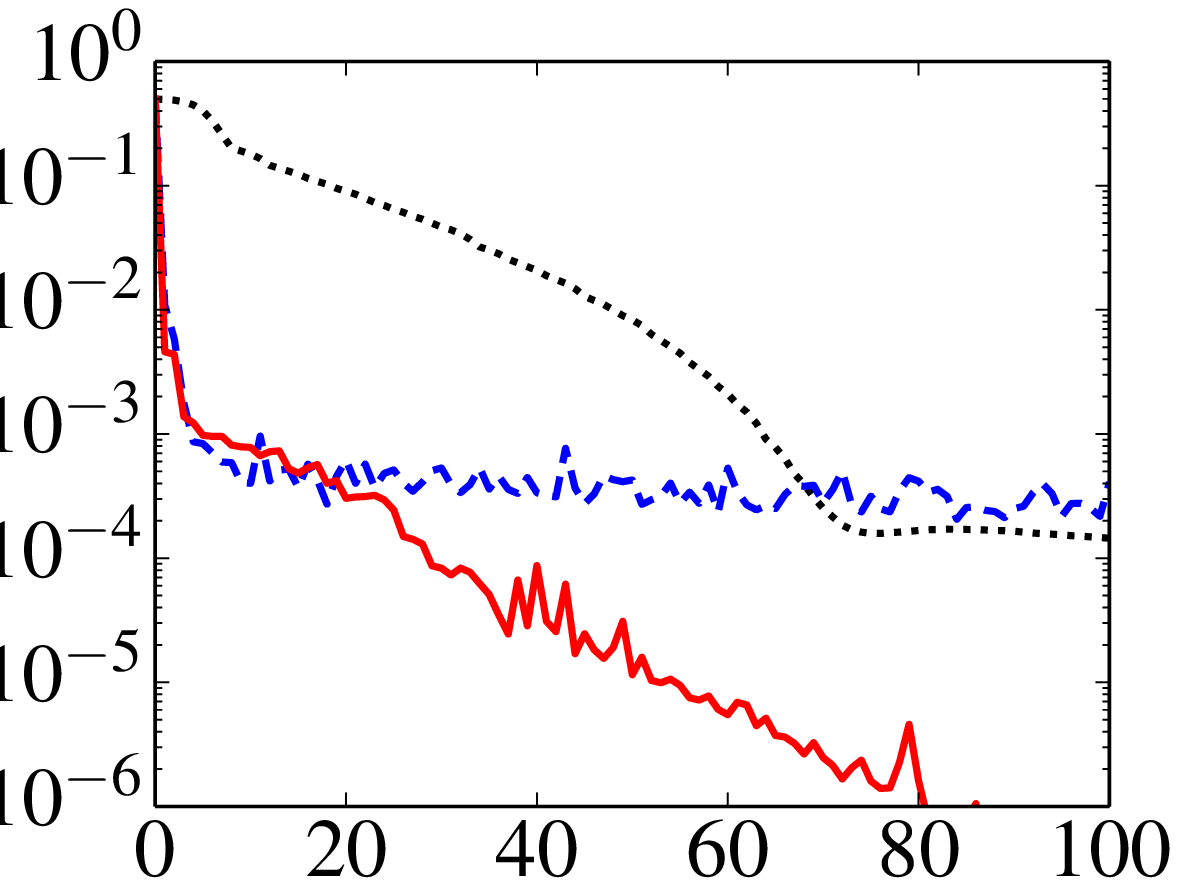}\\
    \raisebox{10ex}{$10^{-8}$} 
        &\quad \includegraphics[width=0.28\textwidth]{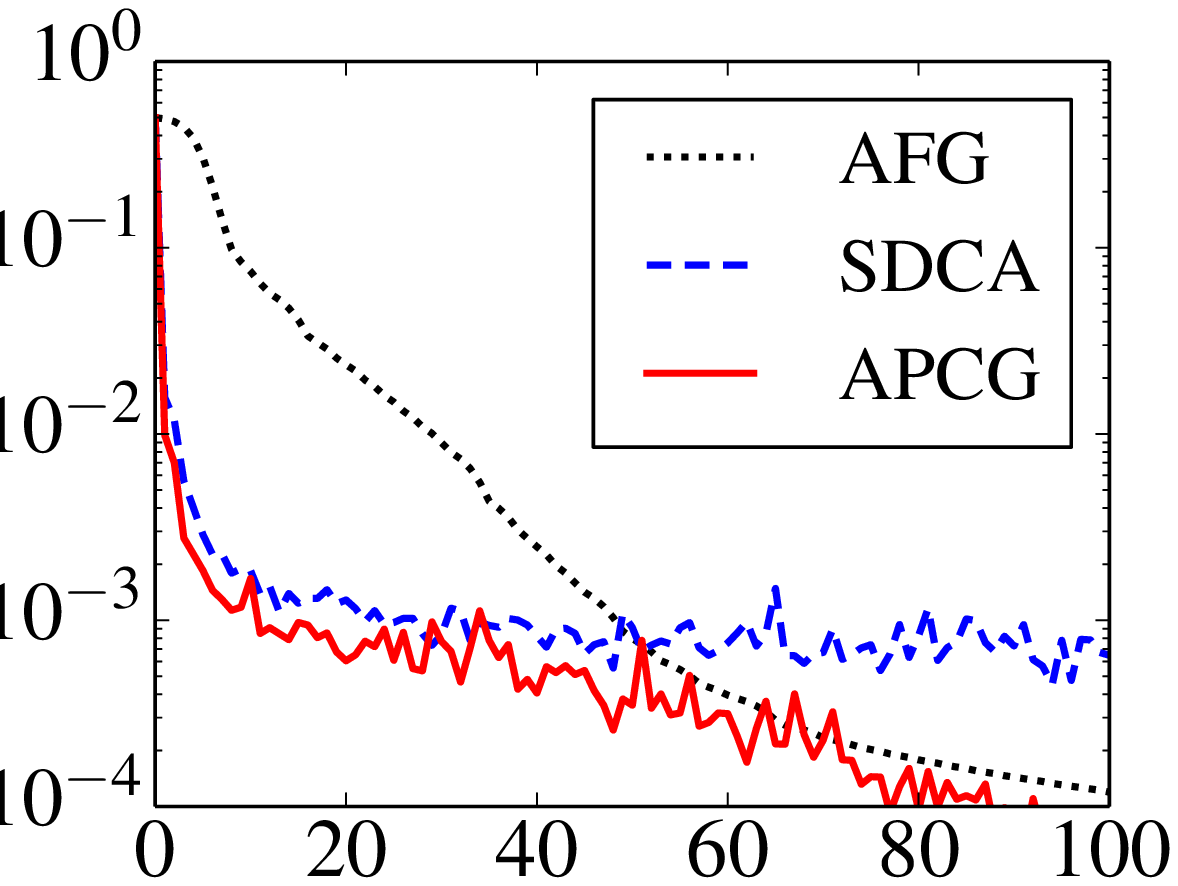} 
        & \includegraphics[width=0.28\textwidth]{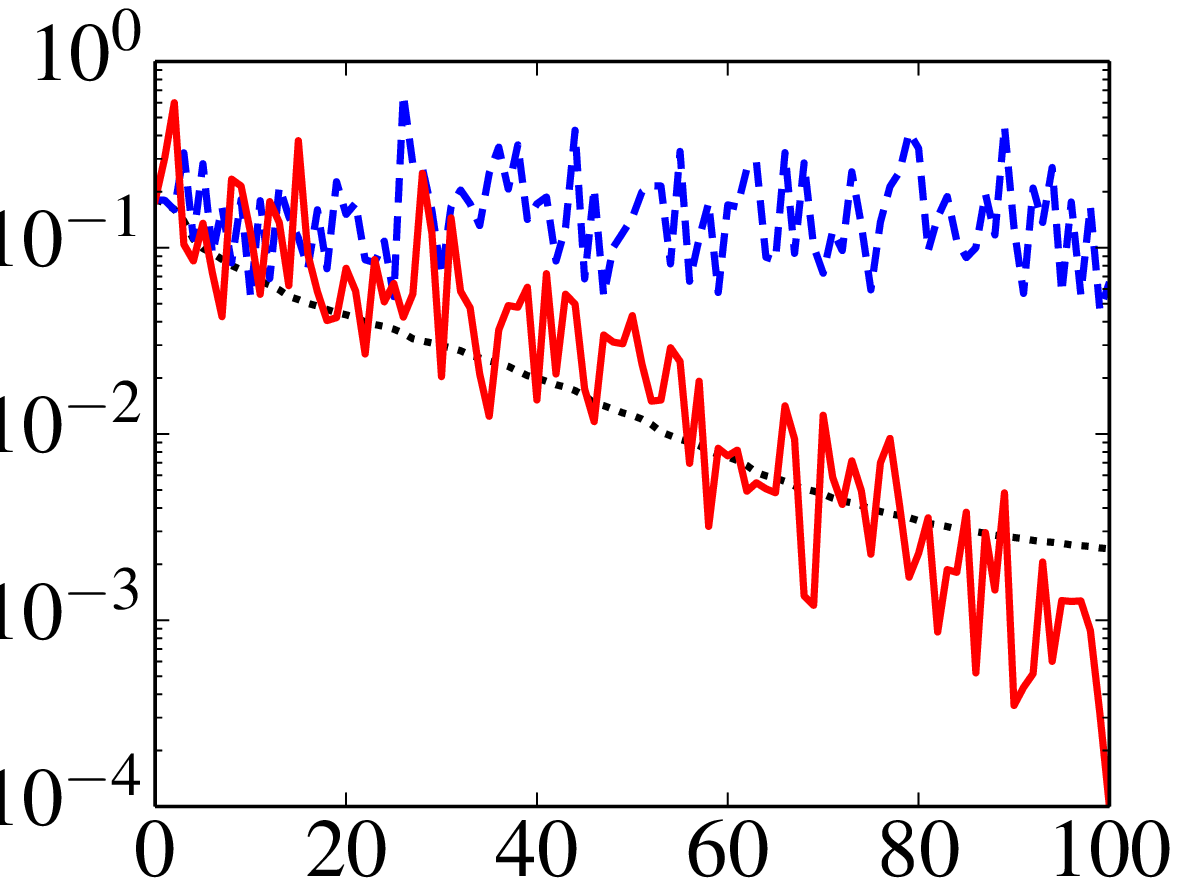} 
        & \includegraphics[width=0.28\textwidth]{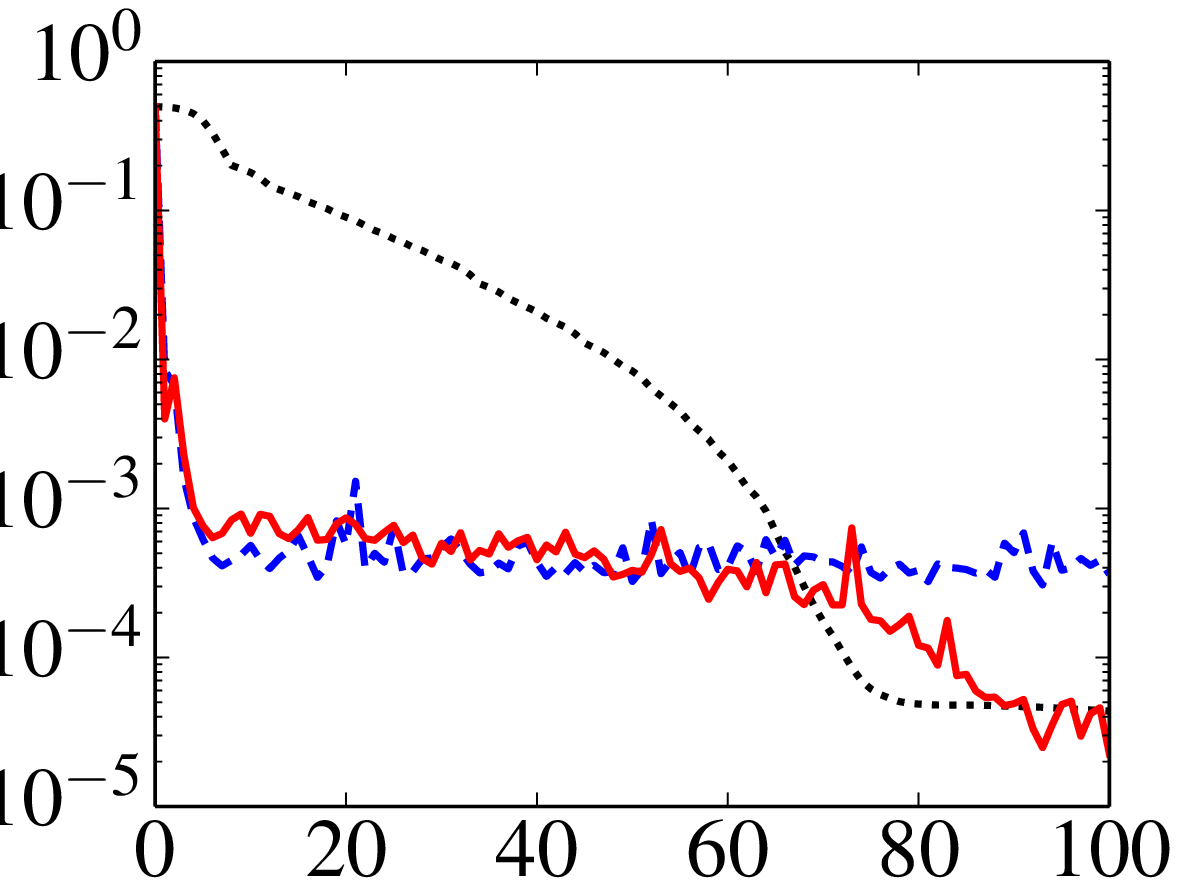}\\
\end{tabular}
\vspace{2ex}
\caption{Comparing the APCG method with SDCA and the accelerated full 
gradient method (AFG).
In each plot, the vertical axis is the primal objective value gap,
i.e., $P(w^{(k)})-P^\star$, and the horizontal axis is the number of passes
through the entire dataset.
The three columns correspond to the three data sets, and each row corresponds
to a particular value of~$\lambda$.
}
\label{fig:primal}
\end{figure}

\begin{figure}[p]
    \begin{tabular}[h]{@{}c|ccc@{}}
    $\lambda$ & RCV1 & covertype & News20 \\
    \hline \\
    \raisebox{10ex}{$10^{-5}$} 
        &\quad \includegraphics[width=0.28\textwidth]{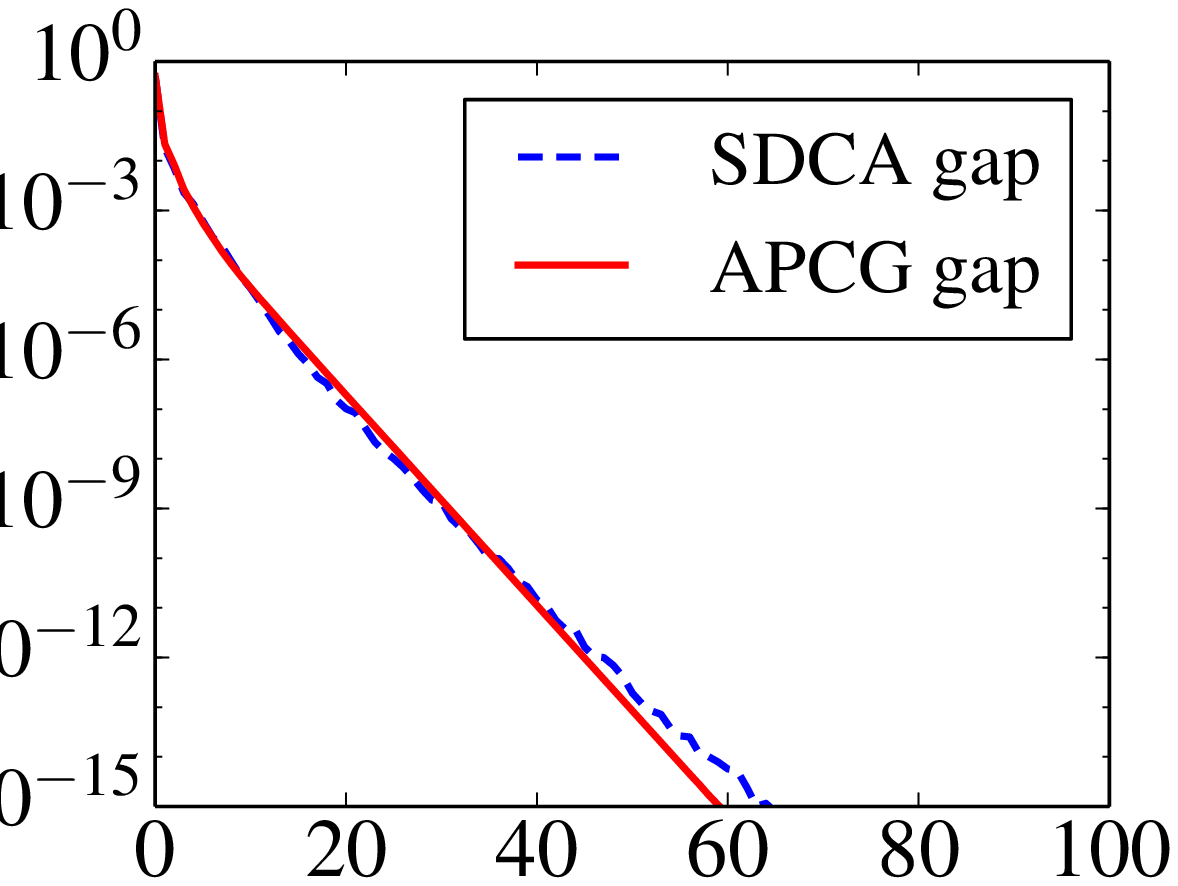} 
        & \includegraphics[width=0.28\textwidth]{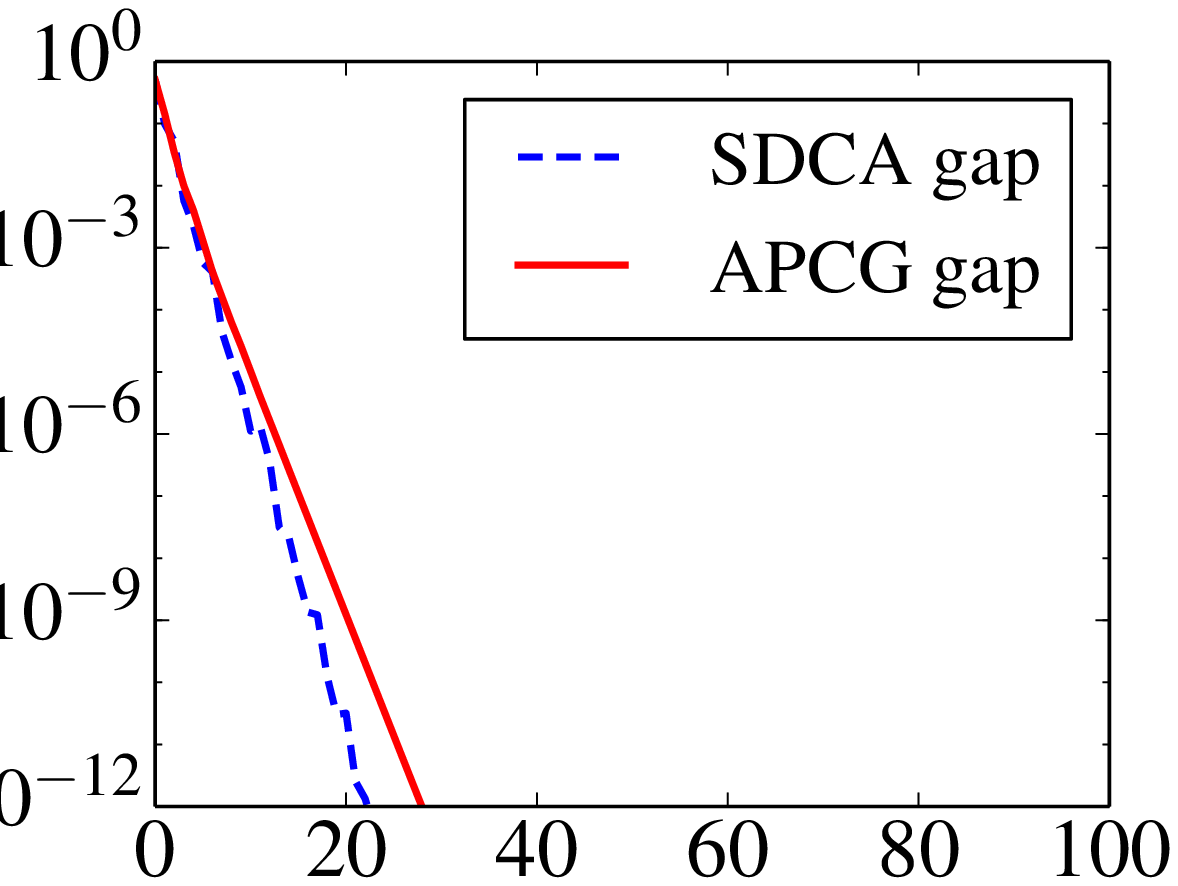} 
        & \includegraphics[width=0.28\textwidth]{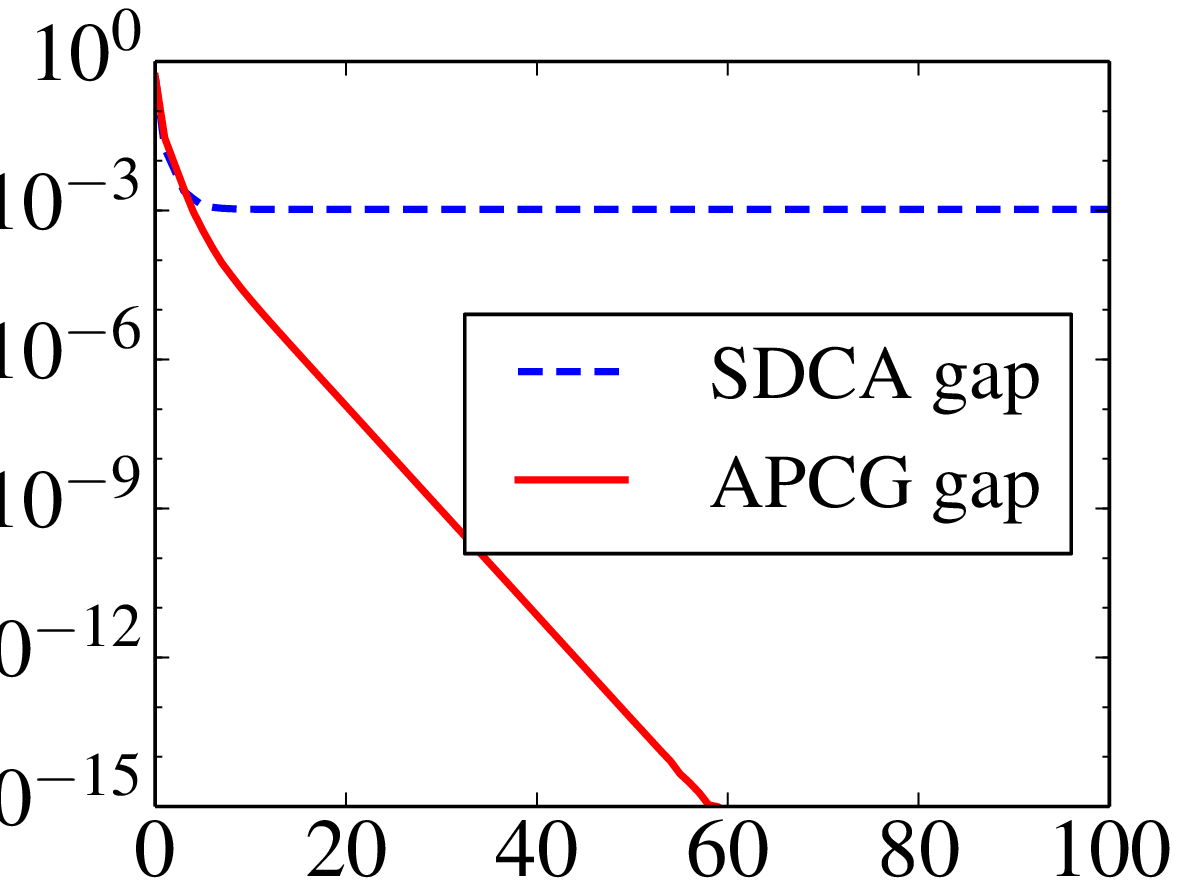}\\
    \raisebox{10ex}{$10^{-6}$} 
        &\quad \includegraphics[width=0.28\textwidth]{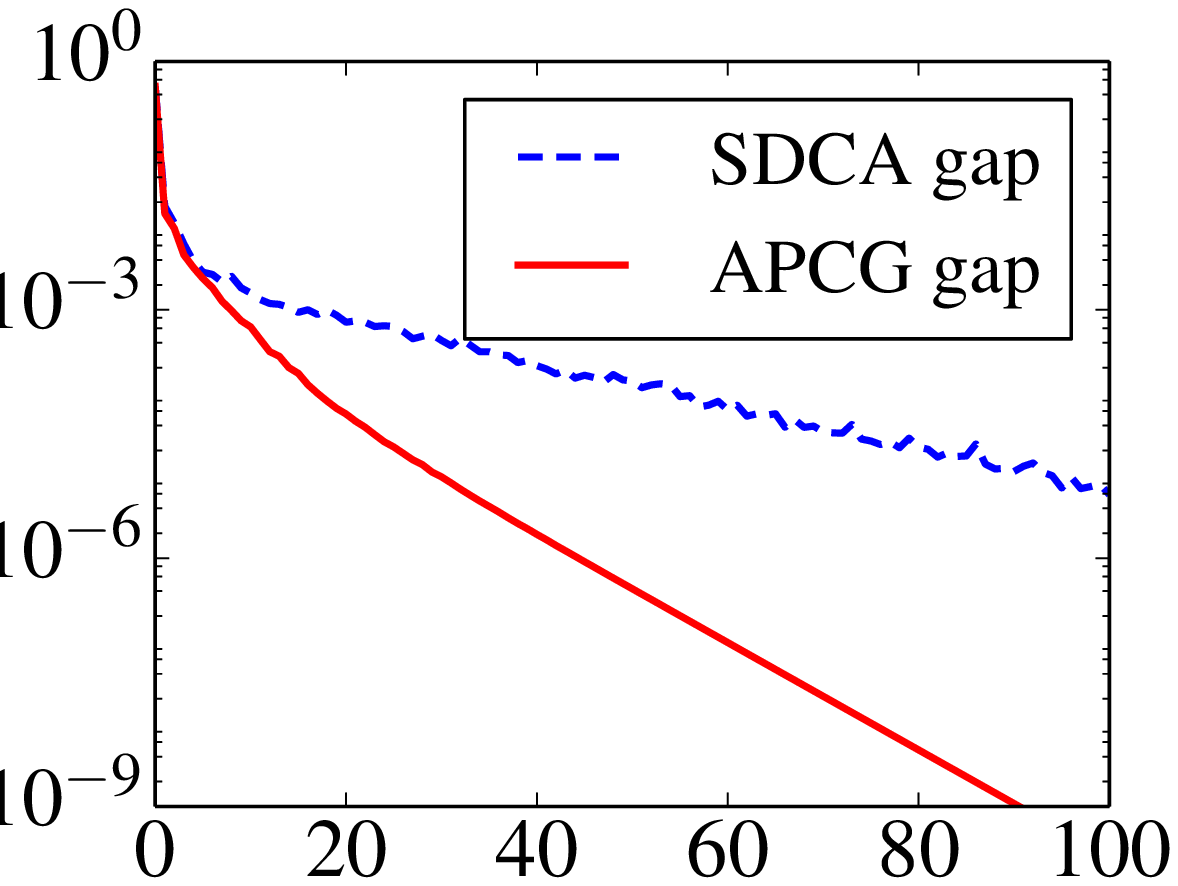} 
        & \includegraphics[width=0.28\textwidth]{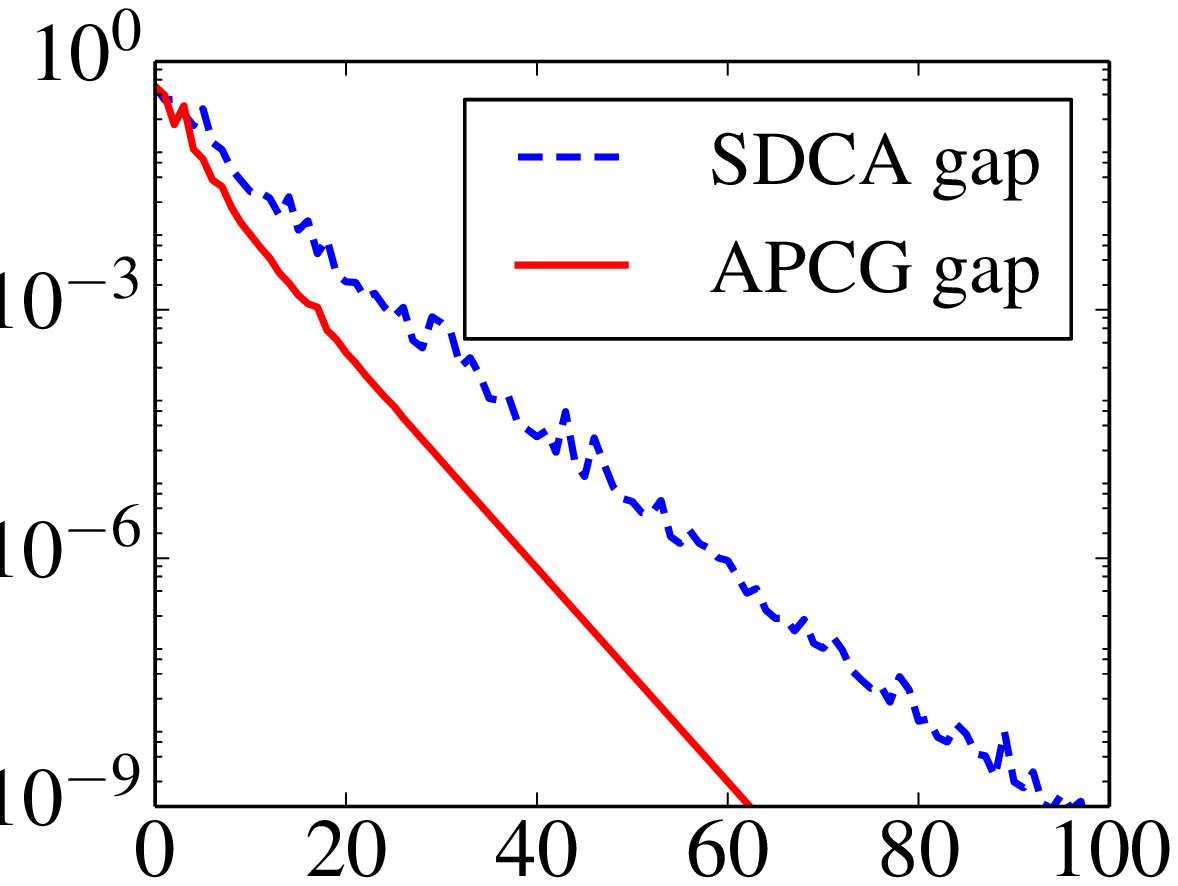} 
        & \includegraphics[width=0.28\textwidth]{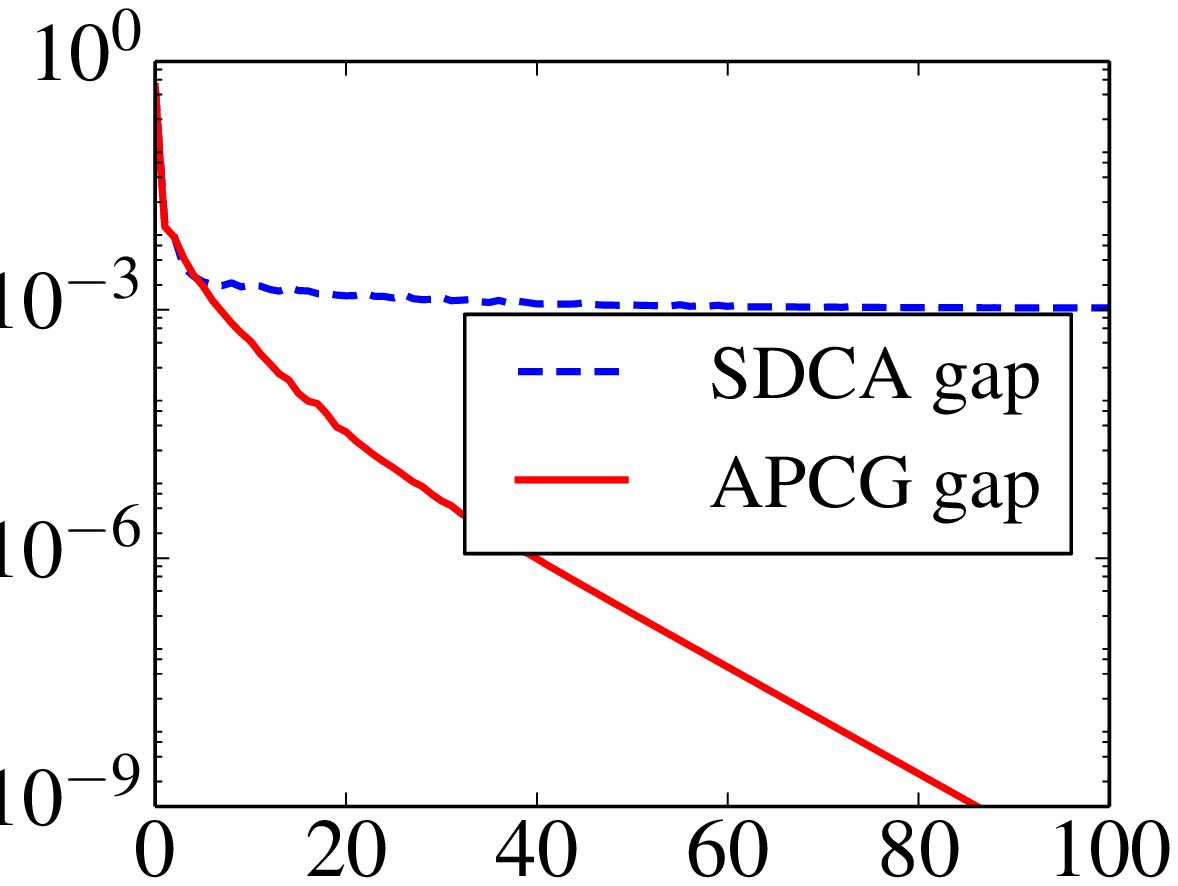}\\
    \raisebox{10ex}{$10^{-7}$} 
        &\quad \includegraphics[width=0.28\textwidth]{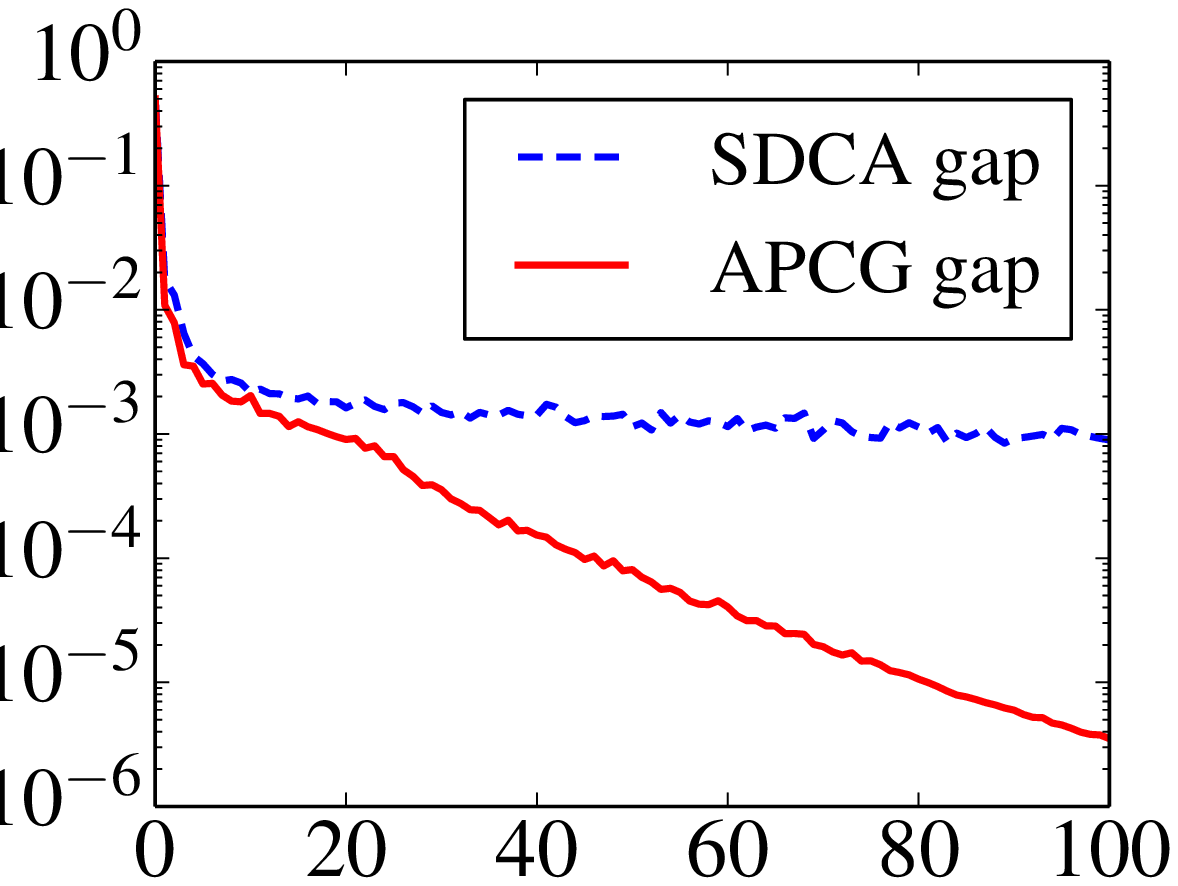} 
        & \includegraphics[width=0.28\textwidth]{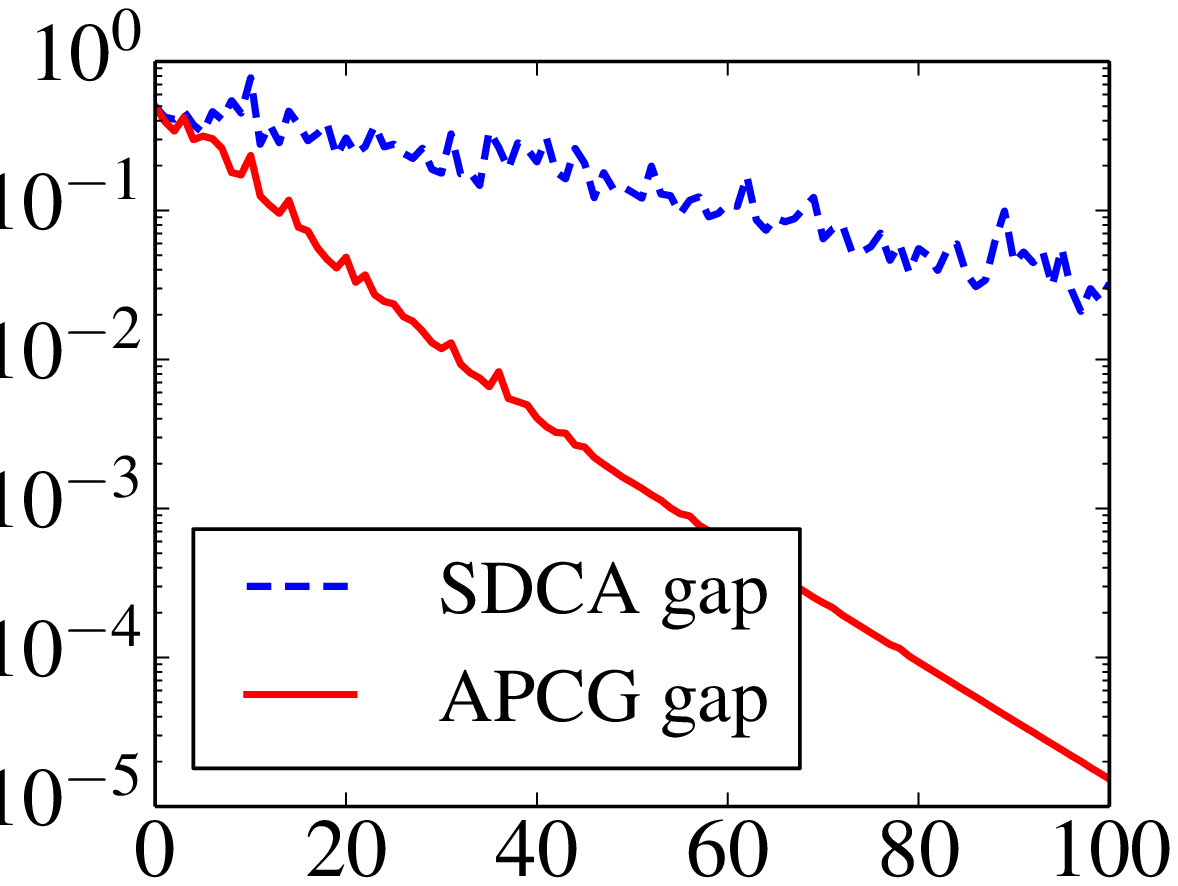} 
        & \includegraphics[width=0.28\textwidth]{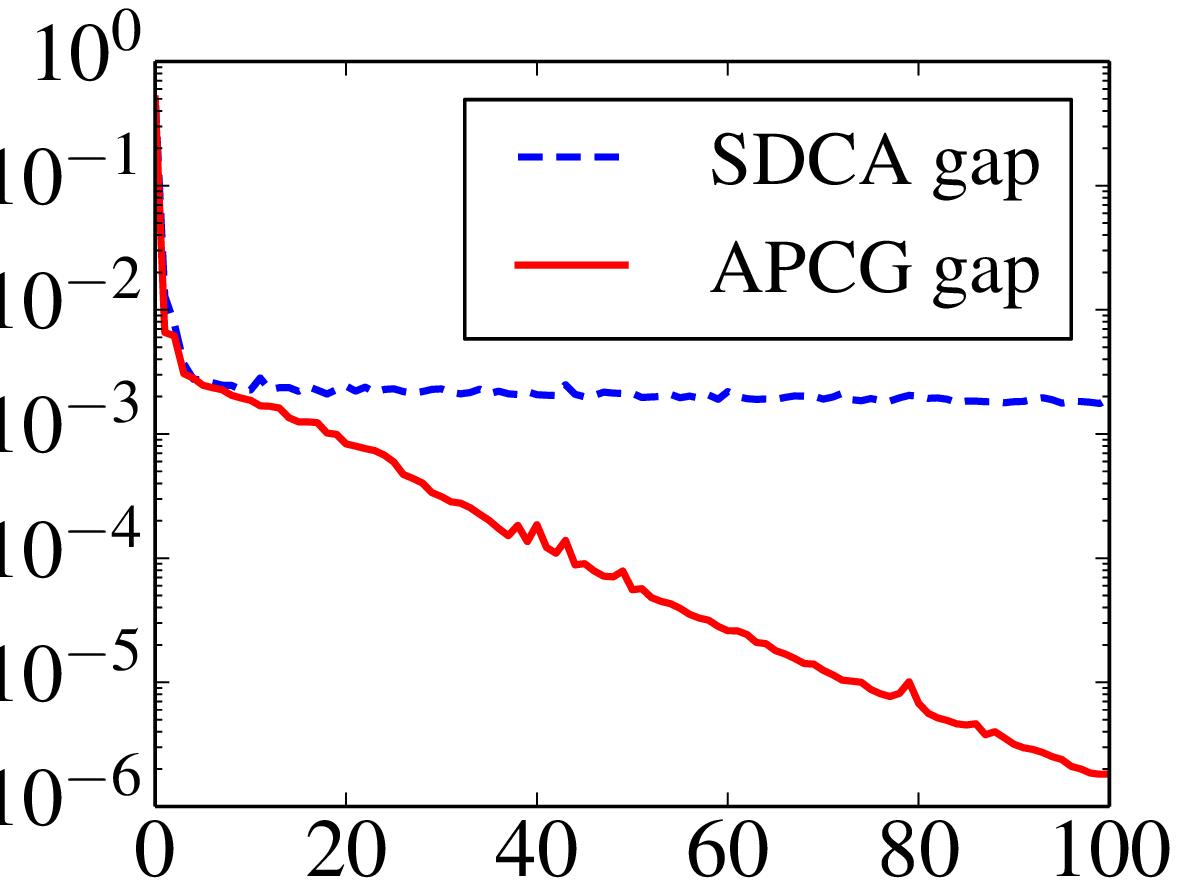}\\
    \raisebox{10ex}{$10^{-8}$} 
        &\quad \includegraphics[width=0.28\textwidth]{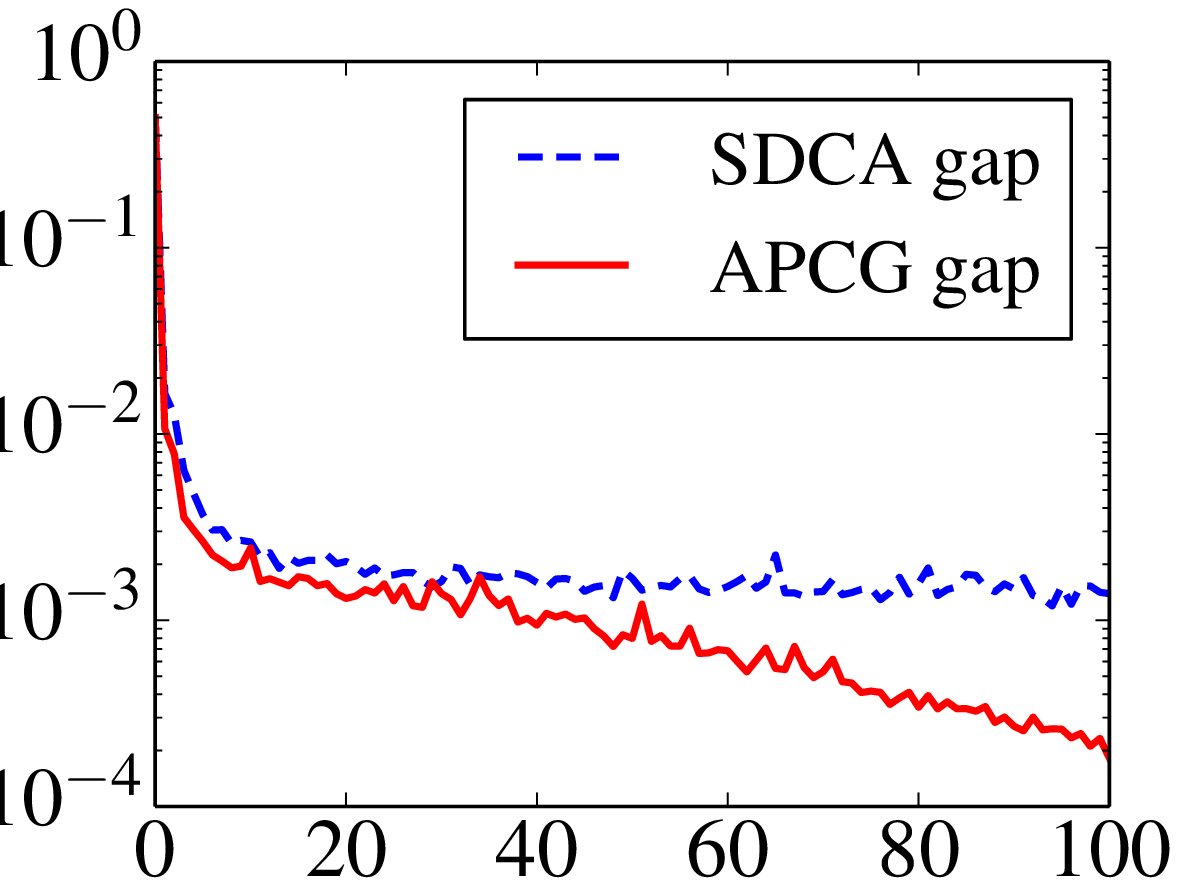} 
        & \includegraphics[width=0.28\textwidth]{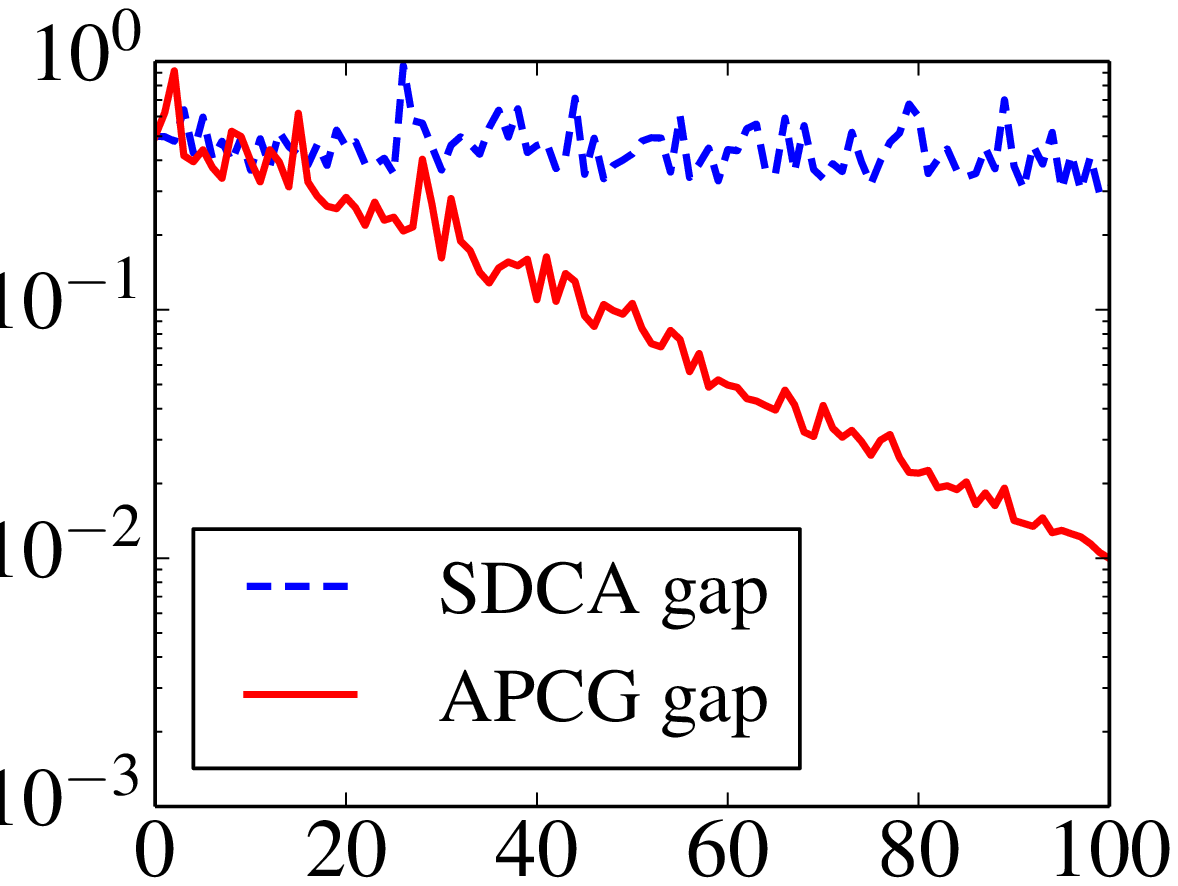} 
        & \includegraphics[width=0.28\textwidth]{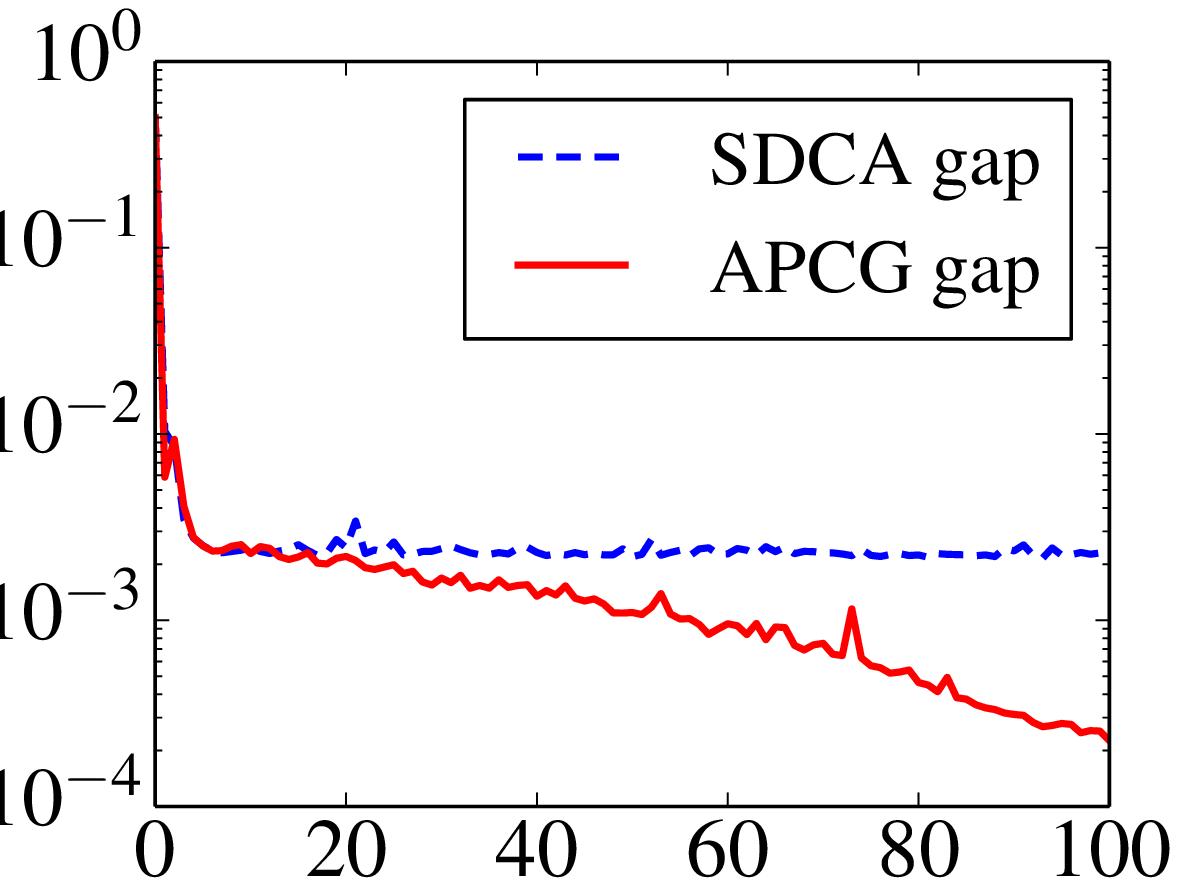}\\
\end{tabular}
\vspace{2ex}
\caption{Comparing the primal-dual objective gap produced by
    APCG and SDCA.
In each plot, the vertical axis is the primal-dual objective value gap,
i.e., $P(w^{(k)})-D(x^{(k)})$, and the horizontal axis is the number of passes
through the entire dataset.
The three columns correspond to the three data sets, and each row corresponds
to a particular value of~$\lambda$.
}
\label{fig:gap}
\end{figure}

\clearpage

\bibliographystyle{plain}
\bibliography{apcg}


\end{document}